\documentclass{amsart}

	\usepackage{amssymb, amsmath, amsthm}
\numberwithin{equation}{section}
\usepackage{cite}                   
\usepackage{paralist}               
\usepackage[margin=1in]{geometry}
\usepackage{appendix}    
\usepackage[final]{graphicx}
\usepackage{subcaption} 
\usepackage{float}
\usepackage{mathrsfs}
\usepackage{array}\usepackage[bookmarks,bookmarksnumbered,bookmarksopen,breaklinks,linktocpage]{hyperref}
\usepackage{nomencl}
\newcommand{\ds}{\displaystyle}
\newcommand{\dd}{\mathrm{d}}

\newcommand{\K}{\mathcal{K}}

\newtheorem{Theorem}{Theorem}[section]
\newtheorem{Proposition}{Proposition}[section]
\newtheorem{Lemma}{Lemma}[section]
\newtheorem{example}{Test case}[section]
\theoremstyle{definition}           
\newtheorem{defn}{Definition}[section]
\theoremstyle{remark}
\newtheorem{Remark}{Remark}[section]

\begin{document}

\title[Convergence-rate and error analysis of sectional-volume average method]{Convergence-rate and error analysis of sectional-volume average method  for the  collisional breakage equation with  multi-dimensional modelling}

\author{Prakrati Kushwah}
\address{Department of Mathematics, National Institute of Technology Tiruchirappalli, Tamil Nadu - 620 015, India.}
\email{kprakrati1256@gmail.com}

\author{Anupama Ghorai}
\address{Department of Mathematics, National Institute of Technology Tiruchirappalli, Tamil Nadu - 620 015, India.}
\email{anupamaghorai@gmail.com}
\author{Jitraj Saha}
\address{Department of Mathematics, National Institute of Technology Tiruchirappalli, Tamil Nadu - 620 015, India.}
\email{jitraj@nitt.edu (Corresponding author)}





\begin{abstract}
	Recent literature reports two sectional techniques, the finite volume method [Das et al., 2020, SIAM J. Sci. Comput., 42(6): B1570-B1598] and the fixed pivot technique [Kushwah et al., 2023, Commun. Nonlinear Sci. Numer. Simul., 121(37): 107244] to solve one-dimensional collision-induced nonlinear particle breakage equation. It is observed that both the methods become inconsistent over  random grids. Therefore, we propose a new birth modification strategy, where the newly born particles are proportionately allocated in three adjacent cells, depending upon the average volume in each cell. This modification technique improves the numerical model by making it  consistent over  random grids. A detailed convergence and error analysis   for this new scheme is studied over different  possible choices of grids such as uniform, nonuniform, locally-uniform, random and oscillatory grids. In addition, we have also identified the conditions upon kernels  for which the  convergence rate increases significantly and the scheme achieves second order of convergence over uniform, nonuniform and locally-uniform grids.	 The enhanced order of accuracy will enable the new model to be easily coupled with CFD-modules. Another significant advancement in the literature is done by extending the discrete model for two-dimensional equation over rectangular grids.
\end{abstract}

\maketitle

	\section{Introduction}
In disperse system,  particles are encountered by several physical processes such as  aggregation, breakage, nucleation, evaporation etc. The particle encounters can be induced by some external force or due to Brownian motion among the fellow particles. During such encounters, particles disintegrate into smaller fragments, thus leading to the breakage process.


Thus, particle volume  evolves in a closed system over  a period of time  and the equations representing these mechanisms are called breakage population balance equations (PBEs). PBEs are initial valued integro-partial differential equations.
Mathematical   model where breakage PBEs are induced by  external force is represented by a  linear equation \cite{saha2023rate}.  On the other hand, mathematical model for breakage PBEs, 
driven by the collisions between particles is  nonlinear equation \cite{MR4566091}. Note that the collisions between two mother particles can lead to an outcome in which one mother particle breaks into smaller fragments while  other remains unchanged, acting as a catalyst.
This study focuses on  
the collision-induced nonlinear breakage process  that has several  applications in real life.
The collision-induced breakage process can be experienced during the formation of raindrops and bubbles.  Prat et al. 
\cite{prat2012influence} have identified that during the formation of raindrop and bubbles, the collisions occur among the raindrops (which can  also be treated as particles) and raindrops disintegrate into smaller ones. They have  modeled this phenomena with the help of  the discrete  nonlinear  equation and
explored the solution (which is the raindrop size distribution) through   Monte Carlo simulations. 
On  a furthernote, the collisional nonlinear breakage equation plays crucial role for productions of powders with particular density distribution for drug formation and  to create specified sizes tablets from wet granulation process \cite{singh2022challenges2}.  Various other applications of such breakage model are found in  dense disperse system for fluidized beds \cite{levenspiel1968processing}, 
astrophysics \cite{piotrowski1953collisions}, the crushing or milling processing, mineral processing, material science and both batch and continuous granulation processes. 
\subsection{The continuous collisional  nonlinear breakage model}
In this article, we consider the continuous collision-induced nonlinear breakage equation which is first introduced by Cheng and Redner \cite{cheng1988scaling,cheng1990kinetics}.  The time evolution  of density distribution function $n(x,t)$ of particles of volume $x\geq0$  at time $t\geq0$ due to  the collision induced breakage process is governed by the following equation:  
\begin{equation}\label{1_1}
	\frac{\partial n(x,t)}{\partial t}=\int_{0}^{\infty} \int_{x}^{\infty} \beta(x|y;z) \mathcal{K}(y,z) n(y,t)n(z,t) \dd y \dd z- \int_{0}^{\infty} \mathcal{K}(x,y)n(x,t)n(y,t) \dd y,
\end{equation}
with the initial condition
\begin{equation}\label{1_2}
	n(x,0)=n_{0}(x)\geq0,\quad\text{for any}\quad x\geq0.
\end{equation}

Here, functions $\beta(x|y;z)$ and $\mathcal{K}(y,z)$ are  the  breakage distribution function and the collisional rate kernel respectively. 
The first term  on the right hand side (R.H.S.) of \eqref{1_1} is the birth term which represents the formation of  $x$ volume particles as a result of  collision between two particles of $y$ $(>x)$ and $z$ volume. The mother particle of $y$ volume  breaks into the smaller particle of volume $x$ whereas particle of  volume $z$  acts like a catalyst and remain unchanged during collision. The second term on R.H.S. of \eqref{1_1}  is the death term that accounts the depletion of particles of volume $x$  from the system.
In general, the collisional  kernel and the breakage distribution function  satisfies the following properties:\\
$(i)$  $\mathcal{K}(y,z)$ is nonnegative and symmetric with respect to the arguments $y$ and $z$ i.e. satisfies the conditions,
\begin{align}\label{1_3}
	0\le \mathcal{K}(y,z)=\mathcal{K}(z,y),\quad\text{for all}\quad y,z>0,
\end{align}
$(ii)$  $\beta(x|y;z)$ is nonnegative  satisfying  the condition  $\beta(x|y;z)=0 $ when $  x> y $ and the volume conserving property,
\begin{align}\label{1_4}
	\int_{0}^{y}x\beta(x|y;z)\dd x=y,\quad\text{for all}\quad  y, z>0,
\end{align}
and total number of fragments produced during a breakage event is calculated as  
\begin{align}\label{1_5}
	\int_{0}^{y}\beta(x|y;z)\dd x=\zeta(y,z),\quad\text{where}\quad2\leq\zeta(y,z)\le\bar{\zeta}<\infty,\quad\text{for all}\quad y,z>0.
\end{align}

For population balance equations, conservation of particles properties such as mass and number of particles of the system plays a key role to  determine the density evolution for certain kinetic rates. In this regard, some integral properties
of the density function play a crucial role to indicate several significant physical properties. 
For this purpose,  the $r^{\text{th}}$ order  moment is denoted by $\mathcal{M}_r(t)$ and  is defined by
\begin{align}\label{1_6}
	\mathcal{M}_r(t)=\int_{0}^{\infty}x^rn(x,t)\dd x \quad\text{for all}\quad t\geq0 \hspace{0.2cm}\text{and}\hspace{0.2cm} r=0,1,2,\dots...
\end{align}
Consider, a function $\varphi(x)$ of positive real numbers. Now, multiplying  the breakage equation \eqref{1_1} with $\varphi(x)$  on both sides and taking  integration over $x$, we obtain  the corresponding  moment equation for the collisional nonlinear breakage equation 
\begin{align}\label{1_7}
	\frac{\dd }{\dd t}\int_{0}^{\infty}\varphi(x)n(x,t)\dd x=\int_{0}^{\infty}\int_{0}^{\infty}\left(\int_{0}^{y}\varphi(x)\beta(x|y;z)\dd x-\varphi(y)\right)\K(y,z)n(y,t)n(z,t)\dd y \dd z.
\end{align}
Setting $\varphi(x)=1$ in  equation \eqref{1_7}  and using property \eqref{1_5}, we obtain the time evolution of zeroth  moment given as
\begin{align}\label{1_9}
	\frac{\dd \mathcal{M}_0(t) }{\dd t}	=\int_{0}^{\infty}\int_{0}^{\infty}\left[\zeta(y,z)-1\right]\K(y,z)n(y,t)n(z,t)\dd y \dd z.
\end{align}
Likewise, setting $\varphi(x)=x$ in  equation  \eqref{1_7} and by property \eqref{1_5},
we obtain  the time evolution of the first moment of the particles as
\begin{align*}
	\frac{\dd \mathcal{M}_1(t) }{\dd t}=0\hspace{0.2cm}\text{implies}\hspace{0.2cm} \mathcal{M}_1(t)=\mathcal{M}_1(0),\hspace{0.2cm}\text{for all}\hspace{0.2cm}t\geq0.
\end{align*}
provided all the integrals in the R.H.S. of equation \eqref{1_7} exists. Thus, for  a suitable choice of kinetic rates the volume conservation law holds appropriately.


\subsection{The state of art and motivation}
The evolution of raindrop  size distribution undergoing the collision induced nonlinear breakage process  \cite{vigil2006destructive,prat2012influence} is extensively used to describe the precipitation dynamics, weather modeling, and radar meteorology.  Moreover, several other applications in real life  and  industrial processes urge researchers  to explore more on the collisional nonlinear breakage equation. In literature, this equation is mostly theoretically analysed \cite{giri2021weak,giri2024continuous} and less numerically investigated.
Since achieving analytical solutions of the collisional nonlinear breakage equation is very challenging, so the primary motivation lies on securing efficient numerical solutions. 
In this regard,  various numerical methods including sectional  discretization methods \cite{MR4666144,MR4566091}, 
Monte Carlo simulations \cite{das2020approximate} and semi-analytical methods \cite{yadav2023homotopy} are implemented in the literature to  solve  the collisional nonlinear breakage equation. Among these numerical methods, the sectional discretization methods are formulated on the basis of distributing the particles in the representatives as well as neighboring  nodes (or pivots) for capturing the  required integral properties accurately. 
Monte Carlo method needs a significant large number of  data to achieve higher accuracy, thus involving a significant computational cost.  On the other hand, semi-analytical method  uses the analytical toolbox of the supporting software and hence very slow to produce results for higher iterations. Additionally, this method fails to give solution of any complex, nonlinear equations in closed form except specific conditions. Thus, the sectional discretization methods are efficient to address and overcome these limitations for generating results with better accuracy. In last few decades, these section based methods includes finite volume methods (FVM)  \cite{kumar2015development, das2024approximate}, fixed pivot techniques (FPT)   \cite{saha2023rate,kumar2008convergenceI} and  cell average techniques (CAT) \cite{kumar2008convergenceII,kushwah2024performance} to solve the linear  breakage equation. In recent literature,  there are  few works documented to  address the numerical treatment of the collisional nonlinear breakage equation, employing  sectional discretization  methods including FVM \cite{MR4666144} and FPT \cite{MR4566091} to solve the equation \eqref{1_1} but these existing schemes have its limitations. For  FVM, the allocation of the new born particles is restricted to a single pivot whereas in FPT,  two pivots are used for birth modification. Moreover, FVM   is highly acccurate for a particular choice of grids that is the  geometric, but for  other choice of grids it struggles with performance. Additionally, FVM fails on coarser grids. On the other hand, FPT offers some definite  improvement, yet it fails  over randomly generated grids. For overall accuracy, it  achieves first order convergence rate.
In this context, we propose an improved sectional discretization method based on particles  averaged volume in a particular cell   named as \emph{volume average method} (VAM). The cell allocation is done in  three nighboring cells based on the properties of particles volume average, which are expected to be preserved. The birth terms are modified in order to achieve consistency with total volume and number of particles. The nonnegative solution obtained from the numerical scheme and consistency of the scheme are examined in detail followed by a discussion on discretization error over different grid types. It is observed that VAM is first order convergent over uniform, nonuniform  and locally-uniform grids. Most significant  observation is that VAM shows  first order convergent over random grids.
We  also study the performance of VAM over  oscillatory grids which is first evidence where consistency analysis is reported for the collisional nonlinear breakage equation in the literature and this new method  achieves first order accurate. An important observation is discussed where  conditions depending upon the kernels for which the VAM shows second order convergence over uniform, nonuniform and locally-uniform grids and maintains first order accuracy over random and  oscillatory grids. The accuracy of the new scheme VAM is validated against the existing fixed pivot technique.   
It is worth mentioning here that the existing FPT  is inconsistent over random grids. Furthermore, we carry out the stability analysis  of the new scheme using Lipschitz criterion. The model is further extended for two-dimensional case over the rectangular grids.  For numerical results, two examples are solved, where the total particle properties and the particle hypervolume are calculated against their exact values.
Importantly, the results show that VAM performs better than the existing FPT.

The article is organized as follows: in section \ref{sec_2}, we describe the  cell adaptive mathematical formulation of VAM for the collisional nonlinear breakage equation \eqref{1_1}. Section \ref{sec_3} discusses  the convergence and consistency analysis. Two-dimensional model is presented in section \ref{sec_4} and  numerical discussion is done for different text problems in section \ref{sec_5}. The findings of the work are highlighted in section \ref{sec_6}.


\section{ The cell volume  average method}\label{sec_2}
We first truncate the continuous collisional breakage equation  \eqref{1_1} by considering  a finite computational domain as $\Lambda=[0,x_{\max}]$ with $x_{\max}<\infty$ : for all $x\in \Lambda$ and $t\geq0$ and the truncated equation \eqref{1_1} reads as
\begin{align}\label{2_1}
	\frac{\partial n(x,t)}{\partial t}=\int_{0}^{x_{\max}} \int_{x}^{x_{\max}} \beta(x|y;z) \mathcal{K}(y,z) n(y,t)n(z,t) \dd y \dd z- \int_{0}^{x_{\max}} \mathcal{K}(x,y)n(x,t)n(y,t) \dd y,
\end{align}
with the initial condition
\begin{align}\label{2_2}
	n(x,0)=n_{0}(x)\quad\text{for any}\quad x\in \Lambda.
\end{align}

We now discretize the volume domain  $\Lambda$ into $I (<\infty)$ discrete cells. Let $\Lambda_i$ denotes the $i^{\text{th}}$ cell, $\Lambda_i:=[x_{i-1/2}, x_{i+1/2}]$ with  $x_{1/2}:=0$ and $x_{I+1/2}:=x_{\max}$ and $\Delta x_i: = x_{i+1/2} - x_{i-1/2}$. A finitely discretized computational domain is shown in Figure \ref{D55}. The midpoint $x_i:= \ds \frac{x_{i+1/2} + x_{i-1/2} }{2}$ of $\Lambda_i$  is the  cell representatives or pivots. It is assumed that the particle property is concentrated at these pivots. 
Consider  	$\Delta x_i\le \Delta x_{i+1}$ and 	$\Delta x_{\text{min}}=\min\limits_{i}\Delta x_i\le\Delta x_i\le \Delta x=\max\limits_{i}\Delta x_i\hspace{0.2cm}\text{with}\hspace{0.2cm}\ds \frac{\Delta x}{\Delta_{\text{min}}}\le \alpha$, a constant, for all $ 1\le i\le I $.
Let $N_i(t)$ be the discrete number density in the $i^{\text{th}}$ cell at time $t$  defined by
\begin{align}\label{2_3}
N_i(t)=\int_{x_{i-1/2}}^{x_{i+1/2}}n(x,t)\dd x.
\end{align} 
Using  relation $\ds n(x,t)\approx \sum_{i=1}^{I}N_i(t) \delta(x-x_i)$  
in the truncated equation \eqref{2_1}, we obtain the collisional  nonlinear breakage equation as
\begin{align}\label{2_4}
\frac{\dd {N}_i(t)}{\dd t} = \hat{ B}_i(t) - \hat{D}_i(t),\quad\text{ with initial conditions}\quad N_i(0)=N_i^0,
\end{align}
for all	$i=1,2,\dots,I$.
Here, birth terms and death terms are defined as respectively
\begin{align}\label{2_5}
		\hat{B}_i(t)=\ds \sum_{k=1}^{I} \sum_{j=i}^{I} \beta_{i,j}^k \K(x_j,x_k) {N}_j(t) {N}_k(t)\quad\text{and}\quad
		\hat{D}_i(t)= \ds \sum_{j=1}^{I} \K(x_i,x_j){ N}_i(t)  {N}_j(t),
\end{align}	with
\begin{align}\label{2_6}
\beta_{i,j}^k= \int_{x_{i-1/2}}^{p_j^i} \beta(x|x_j;x_k)\dd x \quad \text{and} \quad  p_j^i= \left\{\begin{array}{ll} 
	x_i, & \quad \mbox{if}\quad i=j,\\
	x_{i+1/2}, & \quad \mbox{otherwise}.
\end{array}
\right.
\end{align}	
Note that positivity of all the terms appearing on the R.H.S. of $B_i(t)$ and $D_i(t)$ defines positivity  of the both  $B_i(t)$ and $D_i(t)$.
The $r^{th}$ order discrete moment  is 
defined as
$\hat{\mathcal{M}}_r(t)=\ds\sum_{i=1}^{I}x^r_iN_i(t)$, for all $r\in\mathbb{N}$.  Recalling equation \eqref{1_7} the discrete moment equation is written as 
\begin{align}\label{2_7}
\frac{\dd}{\dd t}\sum_{i=1}^{I}\varphi_i N_i(t)=	\sum_{k=1}^{I} \sum_{j=1}^{I}\K(x_j,x_k){ N}_j(t) {N}_k(t)\left(\sum_{i=1}^{j} \int_{x_{i-1/2}}^{p_i^j}\varphi_i \beta(x|x_j;x_k)\dd x-\varphi_j\right),
\end{align}
where $\left\{\varphi_i\right\}_{i\geq1}$ is a sequence of positive real numbers.
Consequently,	 setting $\varphi_i=1$ and $x_i$  in  equation \eqref{2_7},  the time  evolution of the discrete zeroth  moment  and first moment are written as respectively
\begin{align}\label{2_8}
\frac{\dd \hat{\mathcal{M}}_0(t)}{\dd t}=\sum_{k=1}^{I}\sum_{j=1}^{I}\K(x_j,x_k){N}_j(t){N}_k(t)\left[\zeta(x_j,x_k)-1\right],
\end{align}
and
\begin{align}\label{2_9}
\frac{\dd \hat{\mathcal{M}}_1(t)}{\dd t}=\sum_{k=1}^{I}\sum_{j=1}^{I}\K(x_j,x_k){N}_j(t){N}_k(t)\left(\sum_{i=1}^{j} \int_{x_{i-1/2}}^{p_i^j}x_i \beta(x|x_j;x_k)\dd x-x_j\right).
\end{align}
\begin{Remark}
The discrete formulation \eqref{2_4} is not consistent with  discrete  first order moment.\cite{das2020approximate}
\end{Remark}

Since the discrete formulation  \eqref{2_4}  fails to preserve the total mass of the particles in system, it is not suitable  to approximate the collisional  nonlinear breakage equation \eqref{1_1}. To overcome this limitation, we  formulate a new  numerical method which modifies the birth term  of equation \eqref{1_1} and conserves the total mass as well as number of the particles.

\subsection{Cell volume average   based birth rate modification}
To capture the birth rate of the particles more precisely, 
the daughter particles are assigned to the neighboring pivots  depending upon the position of the average volume $\bar{v}_i$ in $\Lambda_i$. In this context, to obtain the average volume,  the	discrete  volume flux over $\Lambda_i$  is defined as 
\begin{align}\label{2_10}
	\hat{V}_i(t)=\sum_{k=1}^{I} \sum_{j=i}^{I}  \K(x_j,x_k) {N}_j(t) {N}_k(t) \int_{x_{i-1/2}}^{p_j^i}x \beta(x|x_j;x_k)\dd x,
\end{align}
and recalling   $B_i(t)$  from equation \eqref{2_5} the volume average $\bar{v}_i$ of all newborn particles in  $\Lambda_i$  as 
\begin{align}\label{2_11}
	\bar{v}_i=\frac{\hat{V}_i(t)}{\hat{B}_i(t)},\quad\text{for all}\quad i=1,2,\dots,I.
\end{align}

There are two possibilities for the particle allocation in a cell as follows:
\begin{enumerate}
	\item (Less likely event) if the  average volume $\bar{v}_i$ of the particles in  $\Lambda_i$  matches with the volume of cell representative (happens very rarely) i.e. $\bar{v}_i=x_i$, then the total birth $B_i$ can be allocated to the node $x_i$. Here, all properties corresponding to  the  average volume of particles are  preserved trivially.
	
	\item (Most likely event) if  the  average volume $\bar{v}_i$ of the particles in $\Lambda_i$ does not match  with the volume of cell representative, that is either $\bar{v}_i>x_i$ or $\bar{v}_i<x_i$, then the particles are distributed to the neighboring pivots such that the total number of particles and mass are preserved.  Then, depending upon the position of average volume, we see the contribution of fractions of birth term in the neighboring pivots.
\end{enumerate}

\begin{figure}[htpb!]
	\centering
	\includegraphics[width=0.9\textwidth]{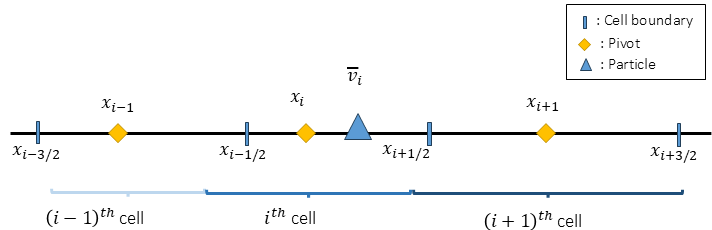}
	\caption{Particle allocation}\label{fig_3}
	\label{D55}
\end{figure}

Consider  $\bar{v}_i>x_i$ (see Figure \ref{fig_3})  and  $x_i$ and $x_{i+1}$ are the two neighboring nodes associated with $\bar{v}_i$. Consider  the terms $c_1(\bar{v}_i,x_i)$ and $c_2(\bar{v}_i,x_{i+1})$ are the fractions of the birth terms $\hat{B}_i(t)$ to be allocate at $x_i$ and $x_{i+1}$ respectively.
Then, to preserve number of particles and volume allocated in $i$ and $i+1$ cell, the  fractions should satisfy the following relations:
\begin{align}\label{2_12}
	c_1(\bar{v}_i,x_i)+c_2(\bar{v}_i,x_{i+1})=\hat{B}_i\quad\text{and}\quad
	x_ic_1(\bar{v}_i,x_i)+ x_{i+1}c_2(\bar{v}_i,x_{i+1})=\bar{v}_i\hat{B}_i.
\end{align}
Solving relations \eqref{2_12}, we get
\begin{align}\label{2_13}
	c_1(\bar{v}_i,x_i)=\hat{B}_i\lambda_i^{+}(\bar{v}_i)\quad\text{and}\quad c_2(\bar{v}_i,x_{i+1})=\hat{B}_i\lambda_{i+1}^{-}(\bar{v}_i),\quad \text{where}\quad\lambda_i^{\pm}(x)= \frac{\ds x-x_{i\pm 1}}{\ds x_i-x_{i \pm 1}}.
\end{align}

Now, there arise four possible birth fractions that can be  considered during a birth assignment at $x_i$. Among these,  two birth fractions arise from the $i^\text{th}$  cell  and other two arise from the neighboring cells (see Figure \ref{fig_4}).

\begin{figure}[htpb!]
	\centering
	\includegraphics[width=0.9\textwidth]{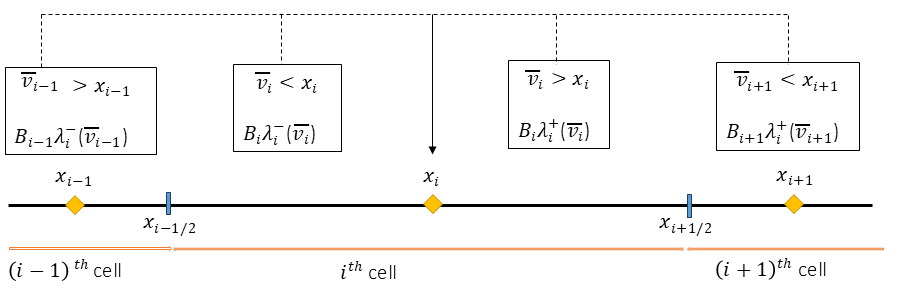}
	\caption{Particle contribution at $x_i$ from all possible cells}\label{fig_4}
	\label{D60}
\end{figure} 

Considering  all possible birth assignments, the following semi-discrete system  obtained due to cell volume average method is defined as
\begin{equation}\label{2_14}
	\frac{\dd \hat{N}_i(t)}{\dd t}=\hat{\mathcal{B}}_i(t)-\hat{\mathcal{D}}_i(t),\hspace{0.1cm}\text{with initial conditions,}\quad\hat{N}_i(0)=\hat{N}_i^0\geq0,
\end{equation}
for all $i=1,2,\dots,I$ and  $\hat{N}_i(t)$ be the solution of \eqref{2_14}. Here, the modified birth term is defined as
\begin{align}\label{2_15} 
	\mathcal{\hat{B}}_i(t)	 
	=&\left\{\begin{array}{ll} 
		\vspace{0.4cm}	b^{(2)}_i+b^{(3)}_{i+1}	, & \quad i=1,\\\vspace{0.4cm}
		b^{(1)}_{i-1} +b^{(2)}_i+b^{(3)}_{i+1}, & \quad i=2,3,\dots,I-1,\\
		b^{(1)}_{i-1} +b^{(2)}_i, & \quad i=I,
	\end{array}
	\right.
\end{align}
where
\begin{align*}
	b^{(1)}_{i-1}:=&\lambda_i^-(\bar{v}_{i-1})H(\bar{v}_{i-1}-x_{i-1}) \hat{B}_{i-1}(t),\quad
	b^{(2)}_{i}:=\lambda_i^-(\bar{v}_{i})H(x_i-\bar{v}_{i})	\hat{B}_{i}(t)+ \lambda_i^+(\bar{v}_{i})H(\bar{v}_{i}-x_{i})\hat{B}_{i}(t),\\
	b^{(3)}_{i+1}:=&\lambda_i^+(\bar{v}_{i+1})H(x_{i+1}-\bar{v}_{i+1})\hat{B}_{i+1}(t)\quad\text{and}\quad H(x)= \left\{\begin{array}{ll} 
		1	, & \quad x > 0,\\
		1/2, & \quad x=0,\\
		0, & \quad x<0,
	\end{array}\hspace{0.1cm} \text{is the Heaviside function}.
	\right.
\end{align*}
The death term  is defined by 
\begin{align}\label{2_16}
	\hat{\mathcal{D}}_i(t)=\sum_{j=1}^{I} \K(x_i,x_j)\hat{ N}_i(t) \hat{ N}_j(t).
\end{align}
Throughout the study we take  the following assumptions as:\\
$(i)$ The collisional kernel satisfies $\mathcal{K}\in\mathcal{C}(\mathbb{R}^2_+)$. Therefore,   for all $y,z\in [x_{\min},x_{\max}]$, there exists a constant $C$  depending on $x_{\max}$ only such that
\begin{align}\label{2_17}
	\sup_{(y,z)\in[x_{\min},x_{\max}]^2}|\mathcal{K}(y,z)|\le C(x_{\max}).
\end{align}
Note that using  the inequality \eqref{2_17}  together with equation \eqref{1_6} in  the discrete equation \eqref{2_14} the  time evolution of the discrete zeroth moment is defined  as
\begin{align}\label{2_18}
	\frac{\dd \mathcal{\hat{M}}_0(t)}{\dd t}\le [\bar{\zeta}-1]C(x_{\max})\mathcal{\hat{M}}^2_0(t).
\end{align}
Thus, the time evolution of the discrete zeroth moment  or the total number of particles  
is bounded on a finite time interval $[0,T]$. So, there exists a constant $\beta(T)$ such that
\begin{align}\label{2_19}
	\mathcal{\hat{M}}_0(t)\le\beta(T),\quad\text{for all}\quad t\in[0,T]. 
\end{align}
\begin{Proposition}\label{prop_2_1}
	The discrete scheme \eqref{2_14}-\eqref{2_16} satisfies the volume conservation law and consistent with the temporal evolution of zeroth moment.
\end{Proposition}
\begin{proof}
	Multiplying $x_i$ to the equation \eqref{2_14} and summing over $i$, we obtain
	\allowdisplaybreaks
	\begin{align}
		\frac{\dd}{\dd t}\sum_{i=1}^{I}x_i\hat{ N}_i(t)=&\sum_{i=1}^{I}x_i\hat{B}_{i-1}(t)\lambda_i^-(\bar{v}_{i-1})H(\bar{v}_{i-1}-x_{i-1}) +\sum_{i=1}^{I}x_i\hat{B}_{i}(t)
		\lambda_i^-(\bar{v}_{i})H(x_i-\bar{v}_{i})\notag \\ 
		&+ \sum_{i=1}^{I}x_i\hat{B}_{i}(t)\lambda_i^+(\bar{v}_{i})H(\bar{v}_{i}-x_{i})+\sum_{i=1}^{I}x_i\hat{B}_{i+1}(t) \lambda_i^+(\bar{v}_{i+1})H(x_{i+1}-\bar{v}_{i+1})\notag\\&-\sum_{i=1}^{I}\sum_{j=1}^{I}x_i\K(x_i,x_j)\hat{ N}_i(t)\hat{ N}_j(t).
		%
		\label{2_20}
	\end{align}
	Using the definition of $\lambda^{\pm}_{i}$ \eqref{2_13} and $H(x)$  in equation \eqref{2_20}, we obtain 
	\begin{align}
		\frac{\dd}{\dd t}\sum_{i=1}^{I}x_i\hat{ N}_i(t)	=&\sum_{i=1}^{I}\bar{v}_i\hat{B}_{i}(t)-\sum_{i=1}^{I}\sum_{j=1}^{I}x_i\K(x_i,x_j)\hat{ N}_i(t)\hat{ N}_j(t)\label{2_21}\\
		=&\sum_{i=1}^{I}\sum_{k=1}^{I} \sum_{j=i}^{I}  \K(x_j,x_k) \hat{N}_j(t)\hat{N}_k(t) \int_{x_{i-1/2}}^{p_j^i}x \beta(x|x_j;x_k)\dd x\notag\\&-\sum_{i=1}^{I}\sum_{j=1}^{I}x_i\K(x_i,x_j)\hat{ N}_i(t)\hat{ N}_j(t).\label{2_22}
	\end{align}
	(see Appendix \ref{appen_A} to  derive equation \eqref{2_21} for detailed calculation).
	Changing order of summations of  the above equation \eqref{2_22}, we get
	\begin{align}
		\frac{\dd}{\dd t}\sum_{i=1}^{I}x_i\hat{ N}_i(t)	=&\sum_{k=1}^{I} \sum_{j=1}^{I}  \K(x_j,x_k)\hat{ N}_j(t)\hat {N}_k(t)\sum_{i=1}^{j} \int_{x_{i-1/2}}^{p_i^j}x \beta(x|x_j;x_k)\dd x\notag\\&-\sum_{i=1}^{I}\sum_{j=1}^{I}x_i\K(x_i,x_j)\hat{ N}_i(t)\hat{ N}_j(t).\label{2_23}
	\end{align}
	Using volume preserving property \eqref{1_4} in equation \eqref{2_23} and simplifying calculation, it yields
	\begin{align*}
		\frac{\dd}{\dd t}\sum_{i=1}^{I}x_i\hat{ N}_i(t)=\sum_{k=1}^{I}\sum_{j=1}^{I}x_j\K(x_j,x_k)\hat{ N}_j(t)\hat{ N}_k(t)-\sum_{i=1}^{I}\sum_{j=1}^{I}x_i\K(x_i,x_j)\hat{ N}_i(t)\hat{ N}_j(t)=0,
	\end{align*}
	that is $\mathcal{\hat{M}}_1(t)=\mathcal{\hat{M}}_1(0)$ for all $t\geq0$, which is the volume conservation law.\\
	The time evolution of the discrete zeroth moment is calculated by taking summation over $i$ on  discrete equation \eqref{2_14} and calculating in similar way as volume conservation and using property \eqref{1_5}
	\begin{align*}
		\frac{\dd}{\dd t}\sum_{i=1}^{I}\hat{N}_i(t)=&\sum_{k=1}^{I} \sum_{j=i}^{I}  \K(x_j,x_k)\hat{N}_j(t)\hat{ N}_k(t)\sum_{i=1}^{j} \int_{x_{i-1/2}}^{p_j^i}\beta(x|x_j;x_k)\dd x-\sum_{i=1}^{I}\sum_{j=1}^{I}\K(x_i,x_j)\hat{ N}_i(t)\hat{ N}_j(t)\notag\\
		=&\sum_{k=1}^{I} \sum_{j=1}^{I}  \K(x_j,x_k)\hat{ N}_j(t) \hat{N}_k(t)\zeta(x_j,x_k)-\sum_{i=1}^{I}\sum_{j=1}^{I}\K(x_i,x_j)\hat{ N}_i(t)\hat{ N}_j(t)\notag\\
		=&\sum_{k=1}^{I}\sum_{j=1}^{I}\K(x_j,x_k)\hat{ N}_j(t)\hat{ N}_k(t)\left[\zeta(x_j,x_k)-1\right].
	\end{align*}
	This is equivalent with the time evolution of discrete zeroth  moment \eqref{2_8}.
\end{proof}
Here, we are a position to introduce the semi-discrete system in a  vector form  described by reformulation\eqref{2_14} as follows: 
$\boldsymbol{\hat{N}}=\left\{\hat{N}_1,\hat{N}_2,\dots,\hat{N}_I\right\}\in\mathbb{R}^I$ is the numerical solution of the following semi-discrete system
\begin{align}\label{2_24}
	\frac{\dd \boldsymbol{\hat{N}}}{\dd t}=&\hat{\boldsymbol{\mathcal{B}}}(\boldsymbol{\hat{N}})-\hat{\boldsymbol{\mathcal{D}}}(\boldsymbol{\hat{N}})
	:=\boldsymbol{\hat{J}}(\boldsymbol{\hat{N}})\quad\text{with}\quad\boldsymbol{\hat{N}}(0)=\boldsymbol{\hat{N}}^0\geq0,
\end{align}
where $i^{th}$ components of $\boldsymbol{\mathcal{\hat{B}}}, \boldsymbol{\mathcal{{\hat{D}}}}\in\mathbb{R}^I$ are defined by $\mathcal{\hat{B}}_i$ \eqref{2_15} and $\mathcal{\hat{D}}_i$ \eqref{2_16} and also the numerical flux  $\boldsymbol{\hat{J}}=\left\{{\hat{J}}_1, {\hat{J}}_2,\dots,{\hat{J}}_I\right\}\in\mathbb{R}^I$ is a non negative vector whose $i^{\text{th}}$ component is given by
${\hat{J}}_i=\hat{\mathcal{B}}_i-\hat{\mathcal{D}}_i.$


\section{Stability and consistency}\label{sec_3}
\subsection{Preliminary definitions and theorems}

Here, we assume the space $X=\mathbb{R}^I$ equipped with 
the discrete $L^1$-norm defined by
\begin{align}\label{3_1}
	||\boldsymbol{\hat{N}}(t)||=\sum_{i=1}^{I}|\hat{N}_i(t)|,\hspace{0.2cm}\text{ for all}\hspace{0.2cm}0\le t\le T.
\end{align}
\begin{defn}
	(Spatial truncation error).  The discretization residue obtained due to the substitution  of the exact solution in the  discrete scheme \eqref{2_24} is called truncation error and is mathematically calculated as
	\begin{align}\label{3_2}
		\sigma(t)=\frac{\dd \boldsymbol{N}(t)}{\dd t} -\frac{\dd \boldsymbol{\hat{N}}(t)}{\dd t},\hspace{0.2cm}\text{ for all}\hspace{0.2cm}0\le t\le T
	\end{align}
	where $\sigma=\left\{\sigma_1,\sigma_2,\dots,\sigma_I\right\}$ is a vector whose $i^\text{th}$ component is given by $\ds \sigma_i(t)=\ds \frac{\dd {N}_i(t)}{\dd t} -\frac{\dd {\hat{N}}_i(t)}{\dd t}$.
\end{defn}
\begin{defn}
	(Global discretization error). A numerical method is called the $p$th order convergent if for $\Delta x\to 0$
	\begin{align}\label{3_3}
		||\boldsymbol{N}(t)-\boldsymbol{\hat{N}}(t)||=\mathcal{O}(\Delta x ^p)\hspace{0.2cm}\text{ for all}\hspace{0.2cm}0\le t\le T.
	\end{align}
\end{defn}
\begin{defn}
	(Lipschitz condition). A mapping $\hat{J}$ which follows
	\begin{align}
		||\boldsymbol{\hat{J}}(\boldsymbol{g})-\boldsymbol{\hat{J}}(\boldsymbol{h})||\le \gamma||\boldsymbol{g}-\boldsymbol{h}||\hspace{0.2cm}\text{ for all}\hspace{0.2cm}\boldsymbol{g},\boldsymbol{h}\in\mathbb{R}^I,
	\end{align}
	is said to be satisfy the Lipschitz condition with $\gamma<\infty$ as Lipschitz constant.
\end{defn}
\begin{defn}\label{def_3_4}
	(Nonnegativity). The system of ODEs in $\mathbb{R}^I$  defined by \eqref{2_24} is called nonnegative or nonnegativity preserving if 
	\begin{align}\label{3_5}
		\boldsymbol{\hat{N}}(0)\geq0 \hspace{0.2cm}\text{implies}\hspace{0.2cm} \boldsymbol{\hat{N}}(t)\geq0\quad\text{for all}\quad 0\leq t\le T.
	\end{align}
\end{defn}
\begin{defn}
	(Consistency). A numerical scheme is called $p^\text{th}$ order consistent if for $\Delta x\to 0$
	\begin{align}\label{3_4}
		||\sigma(t)||=\mathcal{O}(\Delta x ^p)\hspace{0.2cm}\text{ uniformly for all}\hspace{0.2cm}0\le t\le T.
	\end{align}
\end{defn}
\begin{Proposition}\label{pros_3_1}
	Let $\boldsymbol{\hat{J}}\in \mathbb{R}^I$ be the numerical flux defined by the equation \eqref{2_24} in the new scheme VAM. Then $\boldsymbol{\hat{J}}$ satisfies the Lipschitz condition.
\end{Proposition}
\allowdisplaybreaks
\begin{proof}
	Let $\boldsymbol{\hat{M}}, \boldsymbol{\hat{N}}\in \mathbb{R}^I$ be two solutions  satisfying the discrete system \eqref{2_24}. 
	\begin{align}
		||\hat{\boldsymbol{\mathcal{B}}}(\boldsymbol{\hat{M}})-\hat{\boldsymbol{\mathcal{B}}}(\boldsymbol{\hat{N}})||=&\sum_{i=1}^{I}\bigg|\lambda_i^-(\bar{v}_{i-1})H(\bar{v}_{i-1}-x_{i-1}) \left(\hat{B}_{i-1}(\boldsymbol{\hat{M}})-\hat{B}
		_{i-1}(\boldsymbol{\hat{N}})\right)\notag\\&\hspace{0.6cm}+
		\lambda_i^-(\bar{v}_{i})H(x_i-\bar{v}_{i})	\left(\hat{B}_{i}(\boldsymbol{\hat{M}})-\hat{B}_{i}(\boldsymbol{\hat{N}})\right)  
		+\lambda_i^+(\bar{v}_{i})H(\bar{v}_{i}-x_{i})\left(\hat{B}_{i}(\boldsymbol{\hat{M}})-\hat{B}_{i}(\boldsymbol{\hat{N}})\right)\notag\\&\hspace{0.6cm}+ \lambda_i^+(\bar{v}_{i+1})H(x_{i+1}-\bar{v}_{i+1})\left(\hat{B}_{i+1}(\boldsymbol{\hat{M}})-\hat{B}_{i+1}(\boldsymbol{\hat{N}})\right) \bigg|\notag\\
		\le&\sum_{i=1}^{I}\lambda_i^-(\bar{v}_{i-1})H(\bar{v}_{i-1}-x_{i-1}) \left|\hat{B}_{i-1}(\boldsymbol{\hat{M}})-\hat{B}_{i-1}(\boldsymbol{\hat{N}})\right|\notag\\&+
		\sum_{i=1}^{I}\lambda_i^-(\bar{v}_{i})H(x_i-\bar{v}_{i})	\left|\hat{B}_{i}(\boldsymbol{\hat{M}})-\hat{B}_{i}(\boldsymbol{\hat{N}})\right|  
		\notag	\\&	+\sum_{i=1}^{I} \lambda_i^+(\bar{v}_{i})H(\bar{v}_{i}-x_{i})\left|\hat{B}_{i}(\boldsymbol{\hat{M}})-\hat{B}_{i}(\boldsymbol{\hat{N}})\right|\notag\\&+\sum_{i=1}^{I} \lambda_i^+(\bar{v}_{i+1})H(x_{i+1}-\bar{v}_{i+1})\left|\hat{B}_{i+1}(\boldsymbol{\hat{M}})-\hat{B}_{i+1}(\boldsymbol{\hat{N}})\right|.\label{3_7}
	\end{align}
	By the definitions of $\lambda_i^\pm(x)$  and $H(x)$, we have
	\begin{align*}
		0\le \lambda_i^\pm(x)H(x)\le 1. 
	\end{align*}
	Thereby,  estimation \eqref{3_7} can be written as		\begin{align}\label{3_8}
		||\hat{\boldsymbol{\mathcal{B}}}(\boldsymbol{\hat{M}})-\hat{\boldsymbol{\mathcal{B}}}(\boldsymbol{\hat{N}})||\le&\sum_{i=1}^{I}\left|\hat{B}_{i-1}(\boldsymbol{\hat{M}})-\hat{B}_{i-1}(\boldsymbol{\hat{N}})\right|+\sum_{i=1}^{I}\left|\hat{B}_{i}(\boldsymbol{\hat{M}})-\hat{B}_{i}(\boldsymbol{\hat{N}})\right|\notag\\&+\sum_{i=1}^{I}\left|\hat{B}_{i+1}(\boldsymbol{\hat{M}})-\hat{B}_{i+1}(\boldsymbol{\hat{N}})\right|.
	\end{align} 
	Using definition of $\hat{\beta_i}$ and property \eqref{1_5}, bounds of $\K$  \eqref{2_17} and boundedness of the discrete zeroth moment $\mathcal{\hat{M}}_0$ \eqref{2_19}, we deduce that
	\begin{align}
		\left|\hat{\mathcal{B}}_{i}(\boldsymbol{\hat{M}})-\hat{\mathcal{B}}_{i}(\boldsymbol{\hat{N}})\right| =&\sum_{i=1}^{I}\sum_{k=1}^{I}\sum_{j=i}^{I}\K(x_j,x_k)\left| \hat{M}_j \hat{M}_k-\hat{N}_j\hat{N}_k\right|\int_{x_{i-1/2}}^{p_j^i} \beta(x|x_j,x_k)\dd x\notag\\
		\le &\bar{\zeta}C(x_{\max})\sum_{k=1}^{I}\sum_{j=1}^{I}\frac{1}{2}\left|\left(\hat{M}_j+\hat{N}_j\right)\left(\hat{M}_k-\hat{N}_k\right)+\left(\hat{M}_j-\hat{N}_j\right)\left(\hat{M}_K+\hat{N}_k\right)\right|\notag\\
		\le &\frac{\bar{\zeta}C(x_{\max})}{2}\times2\bigg(\sum_{j=1}^{I}\left|\hat{M}_j+\hat{N}_j\right|\bigg)\bigg(\sum_{k=1}^{I}\left|\hat{M}_k-\hat{N}_k\right|\bigg)\notag\\
		\le&2\bar{\zeta}C(x_{\max})\beta(T)||\boldsymbol{\hat{M}}-\boldsymbol{\hat{N}}||.\label{3_9}
	\end{align}
	In similar way, for death term we obtain
	\begin{align}
		||\hat{\boldsymbol{\mathcal{D}}}(\boldsymbol{\hat{M}})-\hat{\boldsymbol{\mathcal{D}}}(\boldsymbol{\hat{N}})||\le 2C(x_{\max})\beta(T)||\boldsymbol{\hat{M}}-\boldsymbol{\hat{N}}||.\label{3_10}
	\end{align}
	Thereby birth term  from\eqref{3_9} and death term from \eqref{3_10}, combinedly we may write
	\begin{align*}
		\left|\left|\left(\hat{\boldsymbol{\mathcal{B}}}(\boldsymbol{\hat{M}})-\hat{\boldsymbol{\mathcal{D}}}(\boldsymbol{\hat{M}})\right)-\left(\hat{\boldsymbol{\mathcal{B}}}(\boldsymbol{\hat{N}})-\hat{\boldsymbol{\mathcal{D}}}(\boldsymbol{\hat{N}})\right)\right|\right|\le& ||\hat{\boldsymbol{\mathcal{B}}}(\boldsymbol{\hat{M}})-\hat{\boldsymbol{\mathcal{B}}}(\boldsymbol{\hat{N}})||+||\hat{\boldsymbol{\mathcal{D}}}(\boldsymbol{\hat{M}})-\hat{\boldsymbol{\mathcal{D}}}(\boldsymbol{\hat{N}})||\notag\\
		\le &2(\bar{\zeta}+1)C(x_{\max})\beta(T)\;||\boldsymbol{\hat{M}}-\boldsymbol{\hat{N}}||\notag\\
		\le &\eta||\boldsymbol{\hat{M}}-\boldsymbol{\hat{N}}||,
	\end{align*}
	where $\eta:=2(\bar{\zeta}+1)C(x_{\max})\beta(T)<\infty$ is a constant, independent of the grids.\\
	Finally, we can write
	\begin{align*}
		||\boldsymbol{\hat{J}}( \boldsymbol{\hat{M}})-\boldsymbol{\hat{J}}( \boldsymbol{\hat{N}})||\le \eta||\boldsymbol{\hat{M}}-\boldsymbol{\hat{N}}||,
	\end{align*}
	which implies that the $\boldsymbol{\hat{J}}$ satisfies the Lipschitz condition irrespective of the grids.
\end{proof}

\subsection{Nonnegativity of the solution}
\begin{Theorem}\label{theo_3_1}
	Let $\boldsymbol{\hat{J}}( \boldsymbol{\hat{N}})$
	be a continuous mapping satisfying Lipschitz condition with respect to $\boldsymbol{\hat{N}}$. Then the semi-discrete system \eqref{2_24} is nonnegative, if and only if for any vector $\boldsymbol{\hat{N}}\in\mathbb{R}^I$ satisfying
	\begin{align}\label{3_11}
		\boldsymbol{\hat{N}}\geq0,\quad\text{with}\quad \hat{N}_i=0\hspace{0.2cm}\text{implies}\hspace{0.2cm}\boldsymbol{\hat{J}(\boldsymbol{\hat{N}})}\geq0,
	\end{align}
	which proves that the solution by the new scheme VAM \eqref{2_14} is nonnegative.
\end{Theorem}
\begin{proof}
	Assume that the  semi-discrete system  \eqref{2_24} is nonnegative. Let $\boldsymbol{\hat{N}}$ be the solution of the system \eqref{2_24}  with initial conditions $\boldsymbol{\hat N(0)}\geq0$. By  definition \ref{def_3_4} of the ODE system, we have $\boldsymbol{\hat N}\geq0$. Now, $\boldsymbol{\hat{N}}\geq0$ with $\hat{N}_i=0$ implies
	$\hat{\mathcal{B}}_i(\boldsymbol{\hat{N}})\geq0\hspace{0.1cm}\text{and}\hspace{0.1cm} \hat{\mathcal{D}}_i(\boldsymbol{\hat{N}})=0 \hspace{0.1cm}\text{for all}\quad i\in\left\{1,2,\dots,I\right\}$.
	So, $\boldsymbol{\hat{N}}\geq0$ with $\hat{N}_i=0$ deduces that ${\hat{J}_i}(\boldsymbol{\hat{N}})\geq0$ for all $i=1,2,\dots,I\hspace{0.1cm} \text{ which implies}\hspace{0.1cm}   \boldsymbol{\hat{J}}(\boldsymbol{\hat{N}})\geq0$.\\
	Conversely, we have  the criteria
	\begin{align*}
		\boldsymbol{\hat{N}}\geq0,\quad\text{with}\quad \hat{N}_i=0\hspace{0.2cm}\text{implies}\hspace{0.2cm}\boldsymbol{\hat{J}(\boldsymbol{\hat{N}})}\geq0.
	\end{align*}
	The above expression implies $\ds  \frac{ \dd {\hat{N}_i}}{\dd t}\ds \geq0.$
	To prove the nonnegativity, we will use the Lipschitz condition on 
	$	\boldsymbol{\hat{J}}$. Since the solution $\boldsymbol{\hat{N}}$ can not cross  the hyperplane $\mathcal{H}_i=\left\{\boldsymbol{\hat{N}}\in\mathbb{R}^I: \hat{N}_i=0\right\}$ so there exists sufficiently small  $\varepsilon>0$ such that
	\begin{align*}
		\boldsymbol{\hat{N}}\geq0,\quad\text{with}\quad \hat{N}_i=0\hspace{0.2cm}\text{implies}\hspace{0.2cm}\frac{\ds \dd{\hat{N}}_i}{\ds \dd t}\geq\varepsilon>0.
	\end{align*}
	This will satisfy the the perturbed ODE system with
	\begin{align*}
		{\tilde{J}}_i={\hat{J}}_i+\varepsilon, \quad\text{for all}\quad i=1,2,\dots,I.
	\end{align*}
	By the Lipschitz condition of $\boldsymbol{\hat{J}}$ and using the standard stability argument for ODEs, if $\varepsilon\to0$, the solution of the unperturbed system with given initial condition will be approximated with exact solution of the perturbed system. This above argument and the criteria \eqref{3_11}, we can write that
	the system \eqref{2_24} is nonnegative. \\
	For the second part, we  have already proved that the system of ODEs \eqref{2_24} is nonnegative in the first part. 
	This directly implies that the solution by the new scheme VAM is nonnegative.
\end{proof}

\subsection{Consistency}
To find the error estimations of  approximated birth term, death term  and total volume flux, we need to integrate the birth term and death term of equation \eqref{2_1}  with  total volume  flux over $\Lambda_i$ and hence define  
\begin{subequations}\label{3_12}
	
	\allowdisplaybreaks
	\begin{align}
		B_i(t):=&\int_{x_{i-1/2}}^{x_{i+1/2}}\int_{0}^{x_{\max}} \int_{x}^{x_{\max}} \beta(x|y;z) \mathcal{K}(y,z) n(y,t)n(z,t) \dd y \dd z \dd x,\\
		D_i(t):=&\int_{x_{i-1/2}}^{x_{i+1/2}}\int_{0}^{x_{\max}} \mathcal{K}(x,y)n(x,t)n(y,t) \dd y \dd x,\\
		V_i(t):=&\int_{x_{i-1/2}}^{x_{i+1/2}}\int_{0}^{x_{\max}} \int_{x}^{x_{\max}} x \beta(x|y;z) \K(y,z) n(y,t)n(z,t) \dd y \dd z \dd x.
	\end{align}
\end{subequations}
\begin{Lemma}\label{lem_3_1}
	Let $B_i$, $D_i$, $V_i$, $\hat{B}_{i}$, $\hat{D}_{i}$, $\hat{V}_{i}$ are defined by the equations \eqref{3_12}, \eqref{2_5} and \eqref{2_10} respectively. Then we have the following error estimates:\\
	\begin{subequations}\label{3_13}
		\begin{flalign}
			&(i) \quad B_i=\hat{B}_{i} + \mathcal{O}(\Delta x_i^3),&&\\
			&(ii)\quad D_i=\hat{D}_{i} + \mathcal{O}(\Delta x_i^3),&&\\
			&(iii)\quad V_i=\hat{V}_{i} + \mathcal{O}(\Delta x_i^3).&&
		\end{flalign}
	\end{subequations}
\end{Lemma}
\begin{proof}
	Considering $B_i(t)$ and the fact that $z$ is independent of $x$ and $y$, so we change the order of integration
	and rearrange the integrals in simplified  discretized form to get
	\begin{align*}
		B_i(t)=&\int_{0}^{x_{I+1/2}}\int_{x_{i-1/2}}^{x_{i+1/2}} \int_{x_{i-1/2}}^{y} \beta(x|y;z) \mathcal{K}(y,z) n(y,t)n(z,t) \dd x \dd y\dd z\\&+\int_{0}^{x_{I+1/2}} \int_{x_{i+1/2}}^{x_{I+1/2}} \int_{x_{i-1/2}}^{x_{i+1/2}}\beta(x|y;z) \K(y,z) n(y,t)n(z,t) \dd x \dd y\dd z.
	\end{align*}
	Applying the midpoint quadrature rule  for first two integrals of above equation, we get
	\allowdisplaybreaks
	\begin{align*}
		B_i(t)=&\sum_{k=1}^{I}\K(x_i,x_k) N_i(t)N_k(t)\int_{x_{i-1/2}}^{x_i} \beta(x|x_i;x_k)  \dd x+\mathcal{O}(\Delta x_i^3)\\&+\sum_{k=1}^{I}\sum_{j=i+1}^{I} \mathcal{K}(x_j,x_k) N_j(t)N_k(t) \int_{x_{i-1/2}}^{x_{i+1/2}} \beta(x|x_j;x_k)\dd x+\mathcal{O}(\Delta x_i^3)\\
		=&\sum_{k=1}^{I}\sum_{j=i}^{I} \K(x_j,x_k) N_j(t)N_k(t) \int_{x_{i-1/2}}^{p^i_{j}} \beta(x|x_j;x_k)\dd x+\mathcal{O}(\Delta x_i^3)\\
		=&\sum_{k=1}^{I}\sum_{j=i}^{I}\beta_{i,j}^k \mathcal{K}(x_j,x_k)N_j(t) N_k(t)+\mathcal{O}(\Delta x_i^3).
	\end{align*}
	Therefore, we have
	\begin{align}\label{3_14}
		B_i(t)=\hat{B}_i(t)+\mathcal{O}(\Delta x_i^3).
	\end{align}
	Similarly, the death term \eqref{3_12} can be written as
	\begin{align}
		D_i(t)=\sum_{j=1}^{I}\K(x_i,x_j)N_i(t)N_j(t)+\mathcal{O}(\Delta x_i^3)
		=\hat{D}_i(t)+\mathcal{O}(\Delta x_i^3),
	\end{align}
	and  the volume flux is written as
	\begin{align*}
		{V}_i(t)=\sum_{k=1}^{I} \sum_{j=i}^{I}  \K(x_j,x_k) N_j(t) N_k(t) \int_{x_{i-1/2}}^{p_j^i}x \beta(x|x_j,x_k)\dd x +\mathcal{O}(\Delta x_i^3)
		=\hat{V}_i(t)+\mathcal{O}(\Delta x_i^3).
	\end{align*}
\end{proof}
We now recall the discrete scheme VAM and simplify  each term in the scheme separately. By  using the definition of $\lambda_i$ and  $\mathcal{\hat{B}}_i(t)$,   we can write the following term as 
\begin{align}
	\lambda_i^-(\bar{v}_{i-1})\hat{B}_{i-1}(t)=\frac{\bar{v}_{i-1}-{x}_{i-1}}{{x}_{i}-{x}_{i-1}}\hat{B}_{i-1}(t)=\frac{2}{\Delta x_i+\Delta x_{i-1}}\left[\bar{v}_{i-1}\hat{B}_{i-1}(t)-{x}_{i-1}\hat{B}_{i-1}(t)\right].\label{3_16}
\end{align}
Substituting the values of $\hat{B}_{i-1}(t)$  \eqref{2_5} and $\bar{v}_{i-1}$ \eqref{2_11}  in  equation \eqref{3_16}, we obtain
\allowdisplaybreaks
\begin{align}
	\lambda_i^-(\bar{v}_{i-1})\hat{B}_{i-1}(t)=&\frac{2}{\Delta x_i+\Delta x_{i-1}}\bigg[\sum_{k=1}^{I} \sum_{j=i-1}^{I}  \K(x_j,x_k)\hat{ N}_j(t)\hat{ N}_k (t)\int_{x_{i-3/2}}^{p_j^{i-1}}x \beta(x|x_j;x_k)\dd x\notag\\
	& \hspace{2.6cm}-x_{i-1}\sum_{k=1}^{I}\sum_{j=i-1}^{I} \mathcal{K}(x_j,x_k) \hat{N}_j(t) \hat{N}_k(t) \int_{x_{i-3/2}}^{p^{i-1}_{j}} \beta(x|x_j;x_k)\dd x\bigg] \notag\\
	=&\frac{2}{\Delta x_i+\Delta x_{i-1}}\bigg[\sum_{k=1}^{I} \sum_{j=i-1}^{I}  \K(x_j,x_k)\hat{ N}_j(t) \hat{N}_k(t) \int_{x_{i-3/2}}^{p_j^{i-1}}(x-x_{i-1}) \beta(x|x_j;x_k)\dd x\bigg]\notag\\
	=&\frac{2}{\Delta x_i+\Delta x_{i-1}}\bigg[\sum_{k=1}^{I}\K(x_{i-1},x_k) \hat{N}_{i-1}(t) \hat{N}_k(t) \int_{x_{i-3/2}}^{x_{i-1}}(x-x_{i-1}) \beta(x|x_{i-1};x_k)\dd x\notag\\
	&\hspace{2.6cm} +\sum_{k=1}^{I} \sum_{j=i}^{I}  \K(x_j,x_k) \hat{N}_j(t)\hat{ N}_k(t) \int_{x_{i-3/2}}^{x_{i-1/2}}(x-x_{i-1}) \beta(x|x_j;x_k)\dd x\bigg]\label{3_17}.
\end{align}
Assume  sufficient smoothness on $\beta$ with respect to $x$ and consider $g(x):=(x-x_{i-1}) \beta(x|x_{i-1};x_k)$.
Applying Taylor series expansions about $x=x_{i-1}$ of function  $g(x)$ having nonzero derivatives upto second order, we obtain
\begin{align}\label{3_18}
g(x)
=g(x_{i-1})+(x-x_{i-1})g'(x_{i-1})
+\mathcal{O}(\Delta x^2_i)=(x-x_{i-1})\beta(x_{i-1}|x_{i-1};x_k)+\mathcal{O}(\Delta x^2_i).
\end{align}
By applying the similar argument on $h(x):=(x-x_{i-1}) \beta(x|x_j;x_k)$, we  obtain
\begin{align}\label{3_19}
h(x)=(x-x_{i-1})\beta(x_{i-1}|x_j;x_k)+(x-x_{i-1})^2\beta_x(x_{i-1}|x_j;x_k)+\mathcal{O}(\Delta x_i^3).
\end{align}
Now substituting the estimates of $g$ and $h$ in \eqref{3_17} and using \eqref{2_3}, it yields
\begin{align*}
\lambda_i^-(\bar{v}_{i-1})\hat{B}_{i-1}(t)=&\frac{2}{\Delta x_i+\Delta x_{i-1}}\bigg[\sum_{k=1}^{I} \K(x_{i-1},x_k) \hat{N}_{i-1}(t) \hat{N}_k(t) \int_{x_{i-3/2}}^{x_{i-1}}g(x)\dd x\\
&\hspace{2.6cm}	+ 
\sum_{k=1}^{I} \sum_{j=i}^{I}  \K(x_j,x_k) \hat{N}_j(t) \hat{N}_k(t)\int_{x_{i-3/2}}^{x_{i-1/2}}h(x)\dd x\bigg]
+\mathcal{O}(\Delta x_i^3)\\
=&\frac{2}{\Delta x_i+\Delta x_{i-1}}\bigg(-\frac{1}{8}\sum_{k=1}^{I}\beta(x_{i-1}|x_{i-1};x_k) \K(x_{i-1},x_k) \hat{N}_{i-1}(t) \hat{N}_k (t)\Delta x^2_{i-1}\\ &\hspace{2.6cm}+
\frac{1}{12}\sum_{k=1}^{I} \sum_{j=i}^{I}\beta_x(x_{i-1}|x_j;x_k)\K(x_j,x_k) \hat{N}_j(t) \hat{N}_k (t)\Delta x^3_{i-1}\bigg)+\mathcal{O}(\Delta x_i^3).
\end{align*}
Let $\bar{h}(x_{i-1})=\beta(x_{i-1}|x_{i-1};x_k) \K(x_{i-1},x_k)N_{i-1}(t)$ and $\bar{g}(x_{i-1})=\beta_x(x_{i-1}|x_j;x_k)$ from the above equation. 
Applying  forwarded difference  on $\bar{h}$ and $\bar{g}$ about $x_i$ in the above equation,
we obtain
\begin{align*}
\lambda_i^-(\bar{v}_{i-1})\hat{B}_{i-1}(t)=&-\frac{\Delta x^2_{i-1}}{\Delta x_i+\Delta x_{i-1}}\times\frac{1}{4}\sum_{k=1}^{I}\beta(x_{i}|x_{i};x_k) \K(x_{i},x_k) \hat{N}_{i}(t) \hat{N}_k (t)\\ &+\frac{\Delta x^3_{i-1}}{\Delta x_i+\Delta x_{i-1}}\times
\frac{1}{6}\sum_{k=1}^{I} \sum_{j=i}^{I}\beta_x(x_{i}|x_j;x_k)\K(x_j,x_k) \hat{N}_j(t) \hat{N}_k (t)+\mathcal{O}(\Delta x_i^3).		
\end{align*}
Similarly, we can obtain the second term as
\begin{align*}
\lambda_i^+(\bar{v}_{i})\hat{B}_{i}(t)=&\frac{\bar{v}_{i}-{x}_{i+1}}{{x}_{i}-{x}_{i+1}}\hat{B}_{i}(t)=\bigg(1-\frac{\bar{v}_{i}-{x}_{i}}{{x}_{i+1}-{x}_{i}}\bigg)\hat{B}_i(t)\\
=&\hat{B}_i(t)+ \frac{\Delta x^2_{i}}{\Delta x_i+\Delta x_{i+1}}
\times	\frac{1}{4}\sum_{k=1}^{I}\beta(x_{i}|x_{i};x_k) \K(x_{i},x_k) \hat{N}_{i}(t) \hat{N}_k (t)\\&-\frac{\Delta x^3_{i}}{\Delta x_i+\Delta x_{i+1}}\times\frac{1}{6}\sum_{k=1}^{I} \sum_{j=i+1}^{I}\beta_x(x_{i}|x_j;x_k)\K(x_j,x_k) \hat{N}_j(t) \hat{N}_k (t)+\mathcal{O}(\Delta x_i^3).
\end{align*}
As in previous way, the third term is
\begin{align*}
\lambda_i^-(\bar{v}_{i})\hat{B}_{i}(t)=&\frac{\bar{v}_{i}-{x}_{i-1}}{{x}_{i}-{x}_{i-1}}\hat{B}_{i}(t)=\bigg(1+\frac{\bar{v}_{i}-{x}_{i}}{{x}_{i}-{x}_{i-1}}\bigg)\hat{B}_i(t)\\
=&\frac{\bar{v}_{i}-{x}_{i+1}}{{x}_{i}-{x}_{i+1}}\hat{B}_{i}(t)=\bigg(1-\frac{\bar{v}_{i}-{x}_{i}}{{x}_{i+1}-{x}_{i}}\bigg)\hat{B}_i(t)\\
=&\hat{B}_i(t)-\frac{\Delta x^2_{i}}{\Delta x_i+\Delta x_{i+1}}\times \frac{1}{4}
\sum_{k=1}^{I}\beta(x_{i}|x_{i};x_k) \K(x_{i},x_k) \hat{N}_{i}(t) \hat{N}_k (t)\\&+\frac{\Delta x^3_{i}}{\Delta x_i+\Delta x_{i+1}}\times\frac{1}{6}\sum_{k=1}^{I} \sum_{j=i+1}^{I}\beta_x(x_{i}|x_j;x_k)\K(x_j,x_k) \hat{N}_j(t) \hat{N}_k (t)+\mathcal{O}(\Delta x_i^3).		
\end{align*}
The fourth term is given by
\begin{align*}
\lambda_i^-(\bar{v}_{i+1})\hat{B}_{i+1}(t)=&\frac{\bar{v}_{i}-{x}_{i+1}}{{x}_{i}-{x}_{i+1}}\hat{B}_{i}(t)\\=&\frac{\Delta x^2_{i+1}}{\Delta x_i+\Delta x_{i+1}}\times
\frac{1}{4}\sum_{k=1}^{I}\beta(x_{i}|x_{i};x_k) \K(x_{i},x_k) \hat{N}_i(t)\hat{N}_k(t)\\ &
-\frac{\Delta x^3_{i+1}}{\Delta x_i+\Delta x_{i+1}}\times\frac{1}{6}\sum_{k=1}^{I} \sum_{j=i+2}^{I}\beta_x(x_{i}|x_j;x_k)\K(x_j,x_k)\hat{N}_j(t)\hat{N}_k(t)+\mathcal{O}(\Delta x_i^3).
\end{align*}
Without loss of generality, the summation appearing in all  the above four expressions can be started from $(i+2)$ since the terms being omitted have $4$th order accuracy.
By  error estimation Lemma \ref{lem_3_1}  and setting  $F(x|y;z):=\beta(x|y;z)\K(y,z)$ and $F_x(x|y;z):= \beta_x(x|y;z)\K(y,z)$, all four expressions rewritten in more simplified form as
\begin{align*}
\lambda_i^-(\bar{v}_{i-1})\hat{B}_{i-1}(t)=&- \frac{\Delta x^2_{i-1}}{\Delta x_i+\Delta x_{i-1}}\times\frac{1}{4}\sum_{k=1}^{I}F(x_{i}|x_{i};x_k)\hat{N}_i(t)\hat{N}_k(t)\\&+\frac{\Delta x^3_{i-1}}{\Delta x_i+\Delta x_{i-1}} \times\frac{1}{6}
\sum_{k=1}^{I} \sum_{j=i+2}^{I}F_x(x_{i}|x_j;x_k)\hat{N}_j(t)\hat{N}_k(t)+\mathcal{O}(\Delta x_i^3),
\end{align*}
\begin{align*}
\lambda_i^+(\bar{v}_{i})\hat{B}_{i}(t)	=&B_i(t)+\frac{\Delta x^2_{i}}{\Delta x_i+\Delta x_{i+1}}\times\frac{1}{4}\sum_{k=1}^{I}F(x_{i}|x_{i};x_k)\hat{N}_i(t)\hat{N}_k(t)\\&-\frac{\Delta x^3_{i}}{\Delta x_i+\Delta x_{i+1}}\times\frac{1}{6} 
\sum_{k=1}^{I} \sum_{j=i+2}^{I}F_x(x_{i}|x_j;x_k)\hat{N}_j(t)\hat{N}_k(t)+\mathcal{O}(\Delta x_i^3),
\end{align*}
\begin{align*}
\lambda_i^-(\bar{v}_{i})\hat{B}_{i}(t)	=&B_i(t)-\frac{\Delta x^2_{i}}{\Delta x_i+\Delta x_{i-1}}\times\frac{1}{4}\sum_{k=1}^{I}F(x_{i}|x_{i};x_k)\hat{N}_i(t)\hat{N}_k(t)\\&+\frac{\Delta x^3_{i}}{\Delta x_i+\Delta x_{i-1}}\times\frac{1}{6} 
\sum_{k=1}^{I} \sum_{j=i+2}^{I}F_x(x_{i}|x_j;x_k)\hat{N}_j(t)\hat{N}_k(t)+\mathcal{O}(\Delta x_i^3)
\end{align*}
and
\begin{align*}
\lambda_i^+(\bar{v}_{i+1})\hat{B}_{i+1}(t)=& \frac{\Delta x^2_{i+1}}{\Delta x_i+\Delta x_{i+1}}\times\frac{1}{4}\sum_{k=1}^{I}F(x_{i}|x_{i};x_k)\hat{N}_i(t)\hat{N}_k(t)\\&-\frac{\Delta x^3_{i+1}}{\Delta x_i+\Delta x_{i+1}}\times\frac{1}{6}
\sum_{k=1}^{I} \sum_{j=i+2}^{I}F_x(x_{i}|x_j;x_k)\hat{N}_j(t)\hat{N}_k(t)+\mathcal{O}(\Delta x_i^3).
\end{align*}
Depending on the position of volume average $\bar{v}_i$, there arise two cases arise: (see  Figure \ref{fig_5})\\
\begin{figure}[htpb!]
\centering
\includegraphics[width=\textwidth]{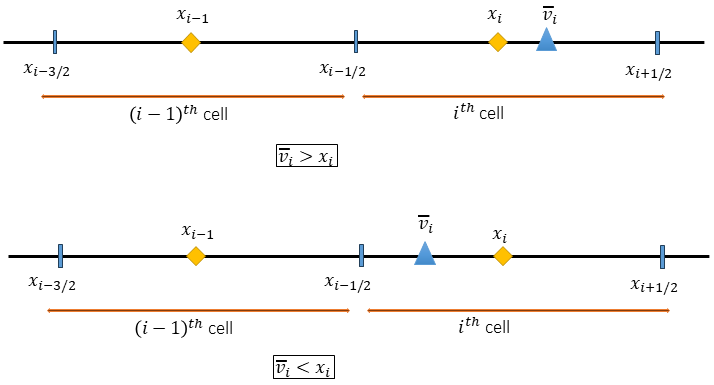}
\caption{$\bar{v}_{i}>x_{i}$ and $\bar{v}_{i}<x_{i}$}\label{fig_5}
\label{D61}
\end{figure}
\textbf{Case I}:
Consider the particle  average volume $\bar{v}_i$ in the  right side of the pivot $x_i$ that is  \\ $\ds\bar{v}_{i-1}> x_{i-1},\; \bar{v}_{i}>x_{i},\; \bar{v}_{i+1}\geq x_{i+1}$. Then 
\begin{align*}
\mathcal{\hat{B}}_{i}(t)=&{B}_{i}(t)+\bigg(\frac{\Delta x^2_{i}}{\Delta x_i+\Delta x_{i+1}}-\frac{\Delta x^2_{i-1}}{\Delta x_i+\Delta x_{i-1}}\bigg)
\times\frac{1}{4}\sum_{k=1}^{I}F(x_{i}|x_{i};x_k)\hat{N}_i(t)\hat{N}_k(t)\\&-\bigg(\frac{\Delta x^3_{i}}{\Delta x_i+\Delta x_{i+1}}-\frac{\Delta x^3_{i-1}}{\Delta x_i+\Delta x_{i-1}}\bigg)\times\frac{1}{6}
\sum_{k=1}^{I} \sum_{j=i+2}^{I}F_x(x_{i}|x_j;x_k)\hat{N}_j(t)\hat{N}_k(t)+\mathcal{O}(\Delta x_i^3).
\end{align*}
\textbf{Case II}:	
Consider the particle  average volume $\bar{v}_i$ in the  left side of the pivot $x_i$  that is\\ $\ds\bar{v}_{i-1}\leq x_{i-1},\; \bar{v}_{i}<x_{i},\; \bar{v}_{i+1}<x_{i+1}$. Then
\begin{align*}
\mathcal{\hat{B}}_{i}(t)=&{B}_{i}(t)+\bigg(\frac{\Delta x^2_{i+1}}{\Delta x_i+\Delta x_{i+1}}-\frac{\Delta x^2_{i}}{\Delta x_i+\Delta x_{i-1}}\bigg)\times
\frac{1}{4}\sum_{k=1}^{I}F(x_{i}|x_{i};x_k) \hat{N}_i(t)\hat{N}_k(t)\\&-\bigg(\frac{\Delta x^3_{i+1}}{\Delta x_i+\Delta x_{i+1}}-\frac{\Delta x^3_{i}}{\Delta x_i+\Delta x_{i-1}}\bigg)\times\frac{1}{6}
\sum_{k=1}^{I} \sum_{j=i+2}^{I}F_x(x_{i}|x_j;x_k)\hat{N}_j(t)\hat{N}_k(t)+\mathcal{O}(\Delta x_i^3).
\end{align*}
\begin{Remark}\label{rem_3_1}
Consider breakage kernel $\beta(x|y;z)=\ds \frac{12x}{y^2}\left(1-\frac{x}{y}\right)$,  therefore $F(x_{i}|x_{i};x_k)=0$. Then,  \textbf{Case I} gives

\begin{align*}
\mathcal{\hat{B}}_{i}(t)= &{B}_{i}(t)-\bigg(\frac{\Delta x^3_{i}}{\Delta x_i+\Delta x_{i+1}}-\frac{\Delta x^3_{i-1}}{\Delta x_i+\Delta x_{i-1}}\bigg)\times\frac{1}{6}
\sum_{k=1}^{I} \sum_{j=i+2}^{I}F_x(x_{i}|x_j;x_k)\hat{N}_j(t)\hat{N}_k(t)+\mathcal{O}(\Delta x_i^3),
\end{align*}
and	 \textbf{Case II} gives
\begin{align*}
\mathcal{\hat{B}}_{i}(t)={B}_{i}(t)-\bigg(\frac{\Delta x^3_{i+1}}{\Delta x_i+\Delta x_{i+1}}-\frac{\Delta x^3_{i}}{\Delta x_i+\Delta x_{i-1}}\bigg)\times\frac{1}{6}
\sum_{k=1}^{I} \sum_{j=i+2}^{I}F_x(x_{i}|x_j;x_k)\hat{N}_j(t)\hat{N}_k(t)+\mathcal{O}(\Delta x_i^3).
\end{align*}
\end{Remark}
In order to investigate the consistency of the scheme, we will use these following results:
\begin{Lemma}\label{lem_3_2}
Let $0<h\in\mathbb{R}^+$ and $B(x,t)\in\mathcal{C}^1(\mathbb{R}^+\times\mathbb{R}^+)$, then for all $v\in\mathbb{R}^+$ we have
\begin{align*}
\int_{v-h}^{v+h}(v-x)B(x,t)\dd x
\begin{cases}
	\geq 0, &\hspace{0.1cm}\mbox{if}\quad B(x,t)\quad\mbox{is monotonically decreasing for}\hspace{0.2cm}x\in[v-h,v+h]\\
	\le 0, &\hspace{0.1cm}\mbox{if}\quad B(x,t)\quad\mbox{is monotonically increasing for}\hspace{0.2cm} x\in [v-h,v+h].
\end{cases}
\end{align*}
\end{Lemma}
\begin{proof}
Using Taylor's series, we have
\begin{align}\label{3_20}
B(x,t)=B(v,t)+(v-x)B_x(\theta,t), \quad\text{for}\quad\theta\in[v-h,v+h].
\end{align}
Therefore, using \eqref{3_20} we have
\begin{align}\label{3_21}
\int_{v-h}^{v+h}(v-x)B(x,t)\dd x=\int_{v-h}^{v+h}(v-x)\left[B(v,t)+ (x-v)B_x(\theta,t)\right]\dd x,\quad\text{for}\quad\theta\in[v-h,v+h].
\end{align}
Since $B(x,t)$ is monotonically decreasing with respect to $x$ for $x\in [v-h,v+h]$, we obtain $B_x(\theta,t)\le 0$. Moreover, evaluating the integrals in equation \eqref{3_21}, we get 
\begin{align*}
\int_{v-h}^{v+h}(v-x)B(x,t)\dd x=-B(v,t)\int_{v-h}^{v+h}(x-v)\dd x- \int_{v-h}^{v+h} (x-v)^2 B_x(\theta,t)\dd x\geq0.
\end{align*}
Proceeding as first part, we can also prove the second part.
\end{proof}
\begin{Lemma}\label{lem_3_3} 
Assume that $\Sigma$ be subset of $\mathbb{R}^+$. 
Let $\beta(x|y;z):\mathcal{C}	(\Sigma^3)\to \mathbb{R}^+$,  $\K(y,z):\mathcal{C}(\Sigma^2)\to \mathbb{R}^+$ and $n(x,t):\mathcal{C}	(\Sigma\times\mathbb{R}^+)\to \mathbb{R}^+$. If the birth rate function 
\begin{align*}
B(x,t)=\int_{0}^{x_{\max}} \int_{x}^{x_{\max}} \beta(x|y;z) \mathcal{K}(y,z) n(y,t)n(z,t) \dd y \dd z,
\end{align*}
has finitely many oscillations (at the most a finite number of maxima and minima) in $\Sigma$ at any time $t$, then the expression $x_i-\bar{v}_i$, $i=1,2,\dots,I$ defined by using the equations \eqref{2_5} and \eqref{2_10}-\eqref{2_11}  
\begin{align*}
x_i-\bar{v}_i=\frac{x_i\hat{B}_i-\hat{V}_i}{\hat{B}_i}, \quad \hat{B}_i>0,
\end{align*}
changes its sign at most finitely many times for $\Delta x$ sufficiently small. 
\end{Lemma}
\begin{proof}
First assume  that $B(x,t)$	has finitely many oscillations (at the most a finite number of maxima and minima) in domain $\Sigma$ at any time $t$. So,  domain $\Sigma$ can be divided into  $m$ sub-domains $\Sigma_1, \Sigma_2,\dots,\Sigma_m$ such that in each sub-domain $\Sigma_k$ for $k=1,2,\dots,m$ the function $B(x,t)$ is either monotonically decreasing or increasing.
By Lemma \ref{lem_3_1}, we obtain

\begin{align}
x_i-\bar{v}_i=&\frac{x_iB_i-V_i+\mathcal{O}(\Delta x_i^3)}{\hat{B}_i}\notag\\
=&  \frac{\ds \int_{x_{i-1/2}}^{x_{i+1/2}}\ds \int_{0}^{x_{\max}}\ds  \int_{x}^{x_{\max}} (x_i-x) \beta(x|y;z) \mathcal{K}(y,z) n(y,t)n(z,t) \dd y \dd z \dd x+\mathcal{O}(\Delta x_i^3)}{\ds \hat{B}_i}\notag\\
=&\frac{\ds \int_{x_{i-1/2}}^{x_{i+1/2}}(x_i-x)B(x,t)\dd x+\mathcal{O}(\Delta x_i^3)}{\ds \hat{B}_i}.
\end{align}
For sufficiently small $\Delta x$ and Lemma \ref{lem_3_2}, we get that $	x_i-\bar{v}_i\geq0$ in any one sub-domain of the sub-domains $\Sigma_1, \Sigma_2,\dots,\Sigma_m$ where the function $B(x,t)$ is monotonically decreasing and $x_i-\bar{v}_i\le0$ in that sub-domain where  the function $B(x,t)$ is monotonically increasing. Since we divide the domain into finite $m$ sub-domains $\Sigma_1, \Sigma_2,\dots,\Sigma_m$, the  change of the sign of the function $x_i-\bar{v}_i$ is finite.
This completes the proof.
\end{proof}
In order to investigate the consistency,	we are now here to provide the discretization error of the collisional nonlinear breakage equation by cell average volume method as follows.

\subsection{Discretization error}
The spatial truncation error is given by for $i=1,2,\dots,I$,
\begin{align}
	\sigma_i(t)=\frac{\dd N_i(t)}{\dd t}-\frac{\dd\hat{ N}_i(t)}{\dd t}=\left[{B}_i(t)-\hat{\mathcal{B}}_i(t)\right]-\left[{D}_i(t)-\hat{\mathcal{D}}_i(t)\right]
\end{align}
We can write
\begin{equation}
	\sigma_i(t)=
	\begin{cases}
		\vspace{0.5cm} \ds C_i\left(\frac{\ds \Delta x^2_{i}}{\ds \Delta x_i+\Delta x_{i+1}}-\frac{\Delta x^2_{i-1}}{\Delta x_i+\Delta x_{i-1}}\right)
		+ \mathcal{O}(\Delta x_i^2), & \quad\mbox{if}\quad \bar{v}_{i-1}>x_{i-1},\;\; \bar{v}_i>x_i,\;\; \bar{v}_{i+1}\geq x_{i+1},\\ \vspace{0.5cm}
		\ds	C_i\left( \frac{\Delta x^2_{i+1}}{\Delta x_i+\Delta x_{i+1}}-\frac{\Delta x^2_{i}}{\Delta x_i+\Delta x_{i-1}}\right)
		+\mathcal{O}(\Delta x_i^2),&\quad\text{if}\quad \bar{v}_{i-1}\le x_{i-1},\;\; \bar{v}_i<x_i,\;\; \bar{v}_{i+1}<x_{i+1},\\ \vspace{0.5cm}
		\ds	\mathcal{O}(\Delta x_i^3),&\quad\mbox{if}\quad \bar{v}_{i-1}\le x_{i-1},\;\; \bar{v}_i=x_i,\;\; \bar{v}_{i+1}\geq x_{i+1},\\ \vspace{0.5cm}
		\ds	\mathcal{O}(\Delta x_i^2),&\quad\mbox{else,}\quad i=1,I_1,I_2,\dots,I_m,I,
	\end{cases}
\end{equation}
where \;
$$C_i=\ds \frac{1}{4}\sum_{k=1}^{I}F(x_{i}|x_{i};x_k)\hat{N}_i(t)\hat{N}_k(t).$$ 
\begin{Remark}\label{rem_3_2}
	Extending the Remark \ref{rem_3_1}, the spatial truncation error is obtained as
	
	\begin{equation}
		\sigma_i(t)=
		\begin{cases}
			\vspace{0.5cm} \ds C'_i\left(\frac{\ds \Delta x^3_{i}}{\ds \Delta x_i+\Delta x_{i+1}}-\frac{\Delta x^3_{i-1}}{\Delta x_i+\Delta x_{i-1}}\right)
			+ \mathcal{O}(\Delta x_i^3), & \quad\mbox{if}\quad \bar{v}_{i-1}>x_{i-1},\;\; \bar{v}_i>x_i,\;\; \bar{v}_{i+1}\geq x_{i+1},\\ \vspace{0.5cm}
			\ds C'_i\left( \frac{\Delta x^3_{i+1}}{\Delta x_i+\Delta x_{i+1}}-\frac{\Delta x^3_{i}}{\Delta x_i+\Delta x_{i-1}}\right)
			+\mathcal{O}(\Delta x_i^3),&\quad\text{if}\quad \bar{v}_{i-1}\le x_{i-1},\;\; \bar{v}_i<x_i,\;\; \bar{v}_{i+1}<x_{i+1},\\ \vspace{0.5cm}
			\ds	\mathcal{O}(\Delta x_i^3),&\quad\mbox{if}\quad \bar{v}_{i-1}\le x_{i-1},\;\; \bar{v}_i=x_i,\;\; \bar{v}_{i+1}\geq x_{i+1},\\ \vspace{0.5cm}
			\ds	\mathcal{O}(\Delta x_i^3),&\quad\mbox{else,}\quad i=1,I_1,I_2,\dots,I_m,I,
		\end{cases}
	\end{equation}\label{3_25_1}
	where 
	$$C'_i=-\ds\frac{1}{6} \sum_{k=1}^{I} \sum_{j=i+2}^{I}F_x(x_{i}|x_j;x_k)\hat{N}_j(t)\hat{N}_k(t).$$
\end{Remark}

\begin{figure}[htpb]
	\centering
	\includegraphics[width=\textwidth]{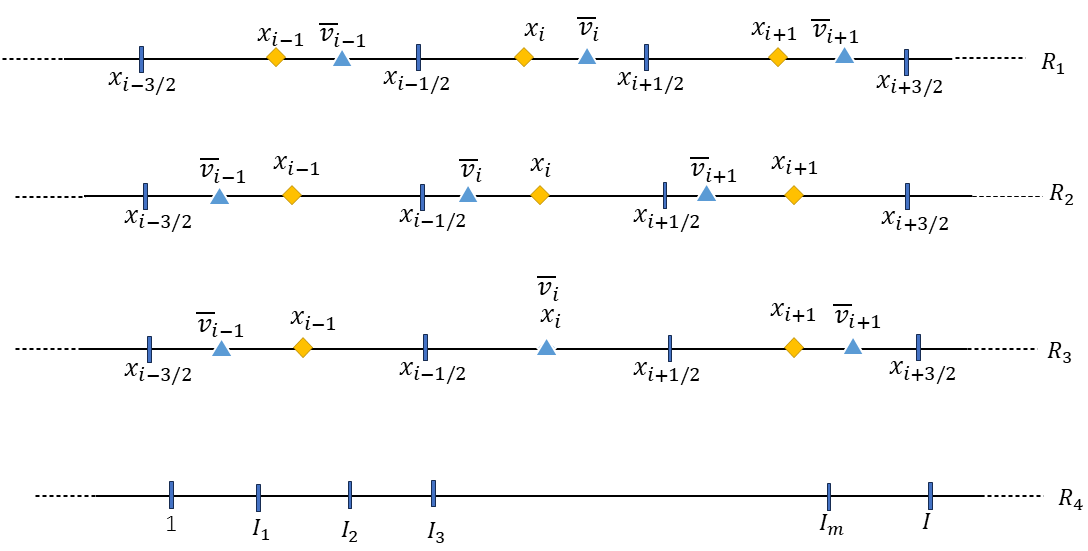}
	\caption{ Four different set of regions }\label{fig_6}
	\label{D52}
\end{figure}

Here, the four cases arise for the order of local error. The last case arises due to the sign
change of $x_i-\bar{v}_i$ and due to outer boundaries. Lemma \ref{lem_3_3} provides that the number of the sign changes of  $x_i-\bar{v}_i$ depends on the properties of the birth rate function $B(x,t)$. Since the sign of  $x_i-\bar{v}_i$ changes in $m$ cells, say, $I_1,I_2,\dots,I_m,$ so the order may deteriorate. Here, the number of cells remains finite, so the it does not lower the order of the numerical scheme.  In simplified form, we describe four cases as follows: (see Figure \ref{fig_6})\\
\begin{align*}
	{R}_1:=&\left\{i\in\mathbb{N}: \bar{v}_{i-1}>x_{i-1},\;\; \bar{v}_i>x_i,\;\; \bar{v}_{i+1}\geq x_{i+1}\right\},\\ 
	{R}_2:=&\left\{i\in\mathbb{N}: \bar{v}_{i-1}\le x_{i-1},\;\; \bar{v}_i<x_i,\;\; \bar{v}_{i+1}<x_{i+1}\right\},\\
	{R}_3:=&\left\{i\in\mathbb{N}: \bar{v}_{i-1}\le x_{i-1},\;\; \bar{v}_i=x_i,\;\; \bar{v}_{i+1}\geq x_{i+1}\right\},\\
	{R}_4:=&\left\{i\in\mathbb{N}: i=1,I_1,I_2,\dots,I_m,I\right\}.
\end{align*}

Then, the order of consistency is 
\begin{align}
	\qquad||\sigma(t)||=&\sum_{i\in{R}_1}|\sigma_i(t)|\sum_{i\in{R}_2}|\sigma_i(t)|+\sum_{i\in{R}_3}|\sigma_i(t)|+\sum_{i\in{R}_4}|\sigma_i(t)|\notag\\
	=&\sum_{i\in{R}_1}\left|C_i\left(\frac{\Delta x^2_{i}}{\Delta x_i+\Delta x_{i+1}}-\frac{\Delta x^2_{i-1}}{\Delta x_i+\Delta x_{i-1}}\right)\right|
	+\sum_{i\in{R}_2}\left|C_i
	\bigg	(\frac{\Delta x^2_{i+1}}{\Delta x_i+\Delta x_{i+1}}-\frac{\Delta x^2_{i}}{\Delta x_i+\Delta x_{i-1}}\bigg)\right|
	\notag\\&
	+|\sigma_1(t)|+\left(|\sigma_{I_1}(t)|+|\sigma_{I_2}(t)|+\dots+|\sigma_{I_m}(t)|\right)+|\sigma_I(t)|+\mathcal{O}(\Delta x^2)\notag\\
	=&\sum_{i\in{R}_1}\left|C_i\left(\frac{\Delta x^2_{i}}{\Delta x_i+\Delta x_{i+1}}-\frac{\Delta x^2_{i-1}}{\Delta x_i+\Delta x_{i-1}}\right)\right|+
	\sum_{i\in{R}_2}\left|C_i
	\bigg	(\frac{\Delta x^2_{i+1}}{\Delta x_i+\Delta x_{i+1}}-\frac{\Delta x^2_{i}}{\Delta x_i+\Delta x_{i-1}}\bigg)\right|
	\notag\\&+
	\mathcal{O}(\Delta x)\label{3_26}.
\end{align}
\subsubsection{Type A: Uniform grids}

\begin{figure}[htpb!]
	\centering
	\includegraphics[width=\textwidth]{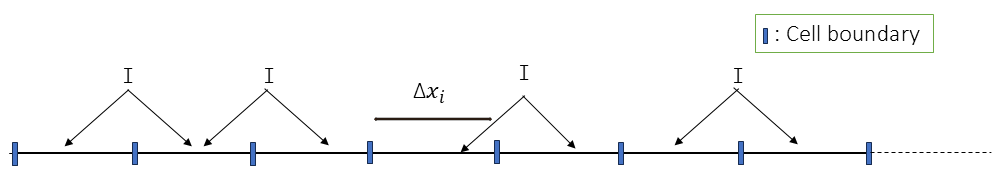}
	\caption{ Uniform  smooth grids }\label{fig_7}
	\label{D53}
\end{figure}

Over uniform grids, shown in Figure \ref{fig_7}, $\Delta x_i=\Delta x$ for $i=1,2,\dots,I.$ From \eqref{3_26} we obtain 
\begin{align}
	||\sigma(t)||=\mathcal{O}(\Delta x).
\end{align} 
\subsubsection{Type B: Nonuniform smooth grids}

Consider a smooth transformation $x=f(\varphi)$ to get nonuniform grid (see in Figure \ref{fig_8}),

\begin{figure}[htpb!]
	\centering
	\includegraphics[width=\textwidth]{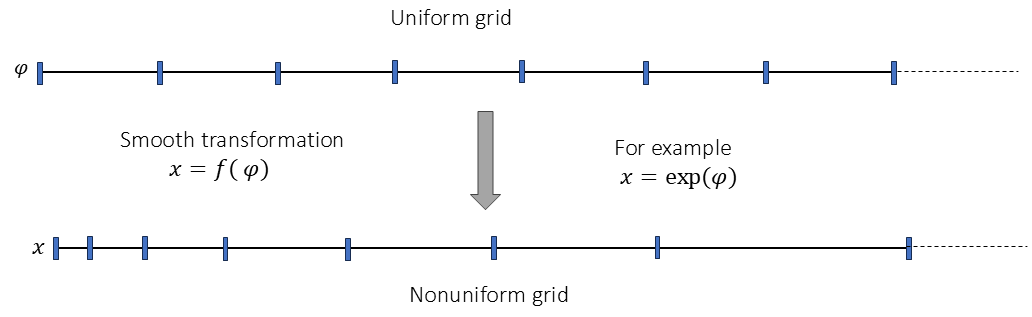}
	\caption{ Nonuniform  smooth grids }\label{fig_8}
	\label{D54}
\end{figure}

where $x_{i\pm1/2}=f(\varphi _i\pm h/2)$ for all $i=1,2,\dots,I$ and $\phi$ is a variable with uniform grids and let $h$ be the uniform grid width in the variable $\phi$. Now expanding the Taylor's series of $f$, we get
\begin{align*}
	\Delta x_i=x_{i+1/2}-x_{i-1/2}=f(\varphi _i+h/2)-f(\varphi _i-h/2)=hf'(\varphi_i)+\frac{h^3}{24}f'''(\varphi_i)+\mathcal{O}(h^4).
\end{align*}
Analogously, we obtain
\begin{align*}
	\Delta x_{i+1}=x_{i+3/2}-x_{i+1/2}=f(\varphi _i+3h/2)-f(\varphi _i+h/2)=hf'(\varphi_i)+h^2f''(\varphi_i)+\mathcal{O}(h^3)
\end{align*}
and
\begin{align*}
	\Delta x_{i-1}=x_{i-1/2}-x_{i-3/2}=f(\varphi _i-h/2)-f(\varphi _i-3h/2)=hf'(\varphi_i)-h^2f''(\varphi_i)+\mathcal{O}(h^3).
\end{align*}
By above approximations we have
\begin{align*}
	\Delta x_i+\Delta x_{i-1}=&2hf'(\varphi_i)-h^2f''(\varphi_i)+\mathcal{O}(h^3),\\
	\Delta x_i+\Delta x_{i+1}=&2hf'(\varphi_i)+h^2f''(\varphi_i)+\mathcal{O}(h^3).
\end{align*}
We can get  from the first term as
\begin{align*}
	\frac{\Delta x^2_{i}}{\Delta x_i+\Delta x_{i+1}}-\frac{\Delta x^2_{i-1}}{\Delta x_i+\Delta x_{i-1}}=\frac{2h^3(f'(\varphi_i))^3-2h^3(f'(\varphi_i))^3+\mathcal{O}(h^4)}{\mathcal{O}(h^2)}=\mathcal{O}(h^2).
\end{align*}
Similarly,
\begin{align*}
	\frac{\Delta x^2_{i+1}}{\Delta x_i+\Delta x_{i+1}}-\frac{\Delta x^2_{i}}{\Delta x_i+\Delta x_{i-1}}=\mathcal{O}(h^2).
\end{align*}
Therefore,
\begin{align*}
	\frac{\Delta x^2_{i}}{\Delta x_i+\Delta x_{i+1}}-\frac{\Delta x^2_{i-1}}{\Delta x_i+\Delta x_{i-1}}=&\mathcal{O}(\Delta x_i^2)\hspace{0.2cm}\text{and}\hspace{0.2cm}
	\frac{\Delta x^2_{i+1}}{\Delta x_i+\Delta x_{i+1}}-\frac{\Delta x^2_{i}}{\Delta x_i+\Delta x_{i-1}}=&\mathcal{O}(\Delta x_i^2).
\end{align*}
From \eqref{3_26} and the above estimates, we obtain
\begin{align*}
	||\sigma(t)||=\mathcal{O}(\Delta x).
\end{align*}
\subsubsection{Type C: Locally-uniform grids}

\begin{figure}[htpb!]
	\centering
	\includegraphics[width=\textwidth]{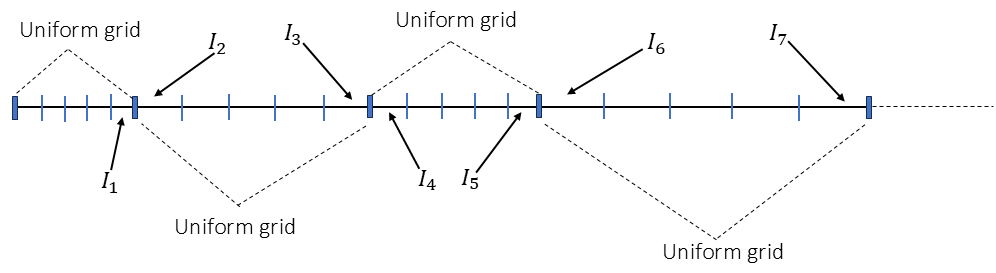}
	\caption{Locally-uniform smooth grids }\label{fig_9}
	\label{D56}
\end{figure}

Over locally-uniform grids, shown in Figure \ref{fig_9},	from \eqref{3_26} we obtain
\begin{align}
	\sigma_i(t)=
	\begin{cases}
		\mathcal{O}(\Delta x_i^2),&\quad\text{for}\quad i=1,I_1,I_2,\dots,I_m,I,\\
		\mathcal{O}(\Delta x_i^3) &\quad\mbox{elsewhere}.
	\end{cases}
\end{align}
Therefore, the order of consistency is given by
\begin{align}
	||\sigma(t)||=
	\mathcal{O}(\Delta x).
\end{align}
\subsubsection{Type D: Random  grids}
Over random grids,  no leading expressions are  going to vanish in \eqref{3_26}. From \textbf{Case I}, we consider first leading term by using the relation $\ds \frac{\Delta x}{\Delta x_{\text{min}}}\leq \alpha$
\begin{align*}
	\frac{\Delta x^2_{i}}{\Delta x_i+\Delta x_{i+1}}C_i
	+\mathcal{O}(\Delta x_i^2)
	\leq \frac{\Delta x^2}{\Delta x_{\text{min}}}+\mathcal{O}(\Delta x_i^2)\leq \alpha \Delta x+\mathcal{O}(\Delta x_i^2).
\end{align*}
Similar argument is applicable for  \textbf{Case II}. 
Therefore, from \eqref{3_26} we can  easily say that the order of consistency is 
\begin{align}
	||\sigma(t)||=\mathcal{O}(\Delta x).
\end{align}
\subsubsection{Type E: Oscillatory grids}
In order to track the oscillatory behavior (unlikely event) in discretized points due to the effect of collisions, let us consider oscillatory grids defined as
\begin{align*}
	\Delta x_{i+1}:=
	\begin{cases}
		\frac{1}{2}\Delta x_i,&\quad\text{if}\quad i \quad\text{is even}\\
		2\Delta x_i &\quad\mbox{if}\quad i\quad\text{is odd}.
	\end{cases}
\end{align*}
By performing this above defined oscillatory grids in \eqref{3_26}, we obtain
\begin{align*}
	||\sigma(t)||=\mathcal{O}(\Delta x).
\end{align*}
Therefore, we can say that VAM is first order consistent on oscillatory grids.
\begin{Remark}\label{rem_3_3}
	It is observed that over uniform, nonuniform and locally-uniform grids, the  VAM executes first order consistency same as the FPT \cite{MR4566091}. Over random grids, VAM exhibits first-order consistency, whereas the FPT becomes inconsistent over random grids  \cite{MR4566091}. This employs that VAM offers better accuracy in terms of consistency compared to  FPT for solving the collisional nonlinear breakage equation over random grids. Furthermore, it is found that VAM achieves first-order consistency by applying to oscillatory grids.
\end{Remark}
\begin{Remark}\label{rem_3_4}
	
	Compiling the Remark \ref{rem_3_1} and \ref{rem_3_2}, we compute the spacial  truncation error over  all four different regions $R_1, R_2, R_3, R_4$. The order of consistency is obtained as
	
	\begin{align}
		\qquad||\sigma(t)||=&\sum_{i\in{R}_1}|\sigma_i(t)|\sum_{i\in{R}_2}|\sigma_i(t)|+\sum_{i\in{R}_3}|\sigma_i(t)|+\sum_{i\in{R}_4}|\sigma_i(t)|\notag\\
		=&\sum_{i\in{R}_1}
		\left|C'_i\left(\frac{\Delta x^3_{i}}{\Delta x_i+\Delta x_{i+1}}-\frac{\Delta x^3_{i-1}}{\Delta x_i+\Delta x_{i-1}}\right)\right|
		+\sum_{i\in{R}_2}
		\left|C'_i\bigg	(\frac{\Delta x^3_{i+1}}{\Delta x_i+\Delta x_{i+1}}-\frac{\Delta x^3_{i}}{\Delta x_i+\Delta x_{i-1}}\bigg)\right|
		\notag\\&
		+|\sigma_1(t)|+\left(|\sigma_{I_1}(t)|+|\sigma_{I_2}(t)|+\dots+|\sigma_{I_m}(t)|\right)+|\sigma_I(t)|+\mathcal{O}(\Delta x^2)\notag\\
		=&\sum_{i\in{R}_1}\left|C'_i\left(\frac{\Delta x^3_{i}}{\Delta x_i+\Delta x_{i+1}}-\frac{\Delta x^3_{i-1}}{\Delta x_i+\Delta x_{i-1}}\right)\right|+
		\sum_{i\in{R}_2}\left|C'_i
		\bigg	(\frac{\Delta x^3_{i+1}}{\Delta x_i+\Delta x_{i+1}}-\frac{\Delta x^3_{i}}{\Delta x_i+\Delta x_{i-1}}\bigg)\right|
		\notag\\&+
		\mathcal{O}(\Delta x^2)\label{3_31}.
	\end{align}
	Thus VAM shows second order consistency that is  $||\sigma(t)||=\mathcal{O}(\Delta x^2)$ over  uniform, nonuniform and locally-uniform grids, whereas $||\sigma(t)||=\mathcal{O}(\Delta x)$ over random and oscillatory grids.
	
\end{Remark}
\subsection{Convergence}
To prove the  consistency and convergence of the new scheme VAM, the following theorem is utilized.
\begin{Theorem}\label{theo_2_2}[Convergence theorem]
	Assume that the mapping $\boldsymbol{\hat{J}}$  satisfies  Lipschitz condition for $0\le t\le T$ and for all $\boldsymbol{N}, \boldsymbol{\hat{N}}\in\mathbb{R}^I$ where $\boldsymbol{N}$ and $\boldsymbol{\hat{N}}$ are projected exact and numerical solutions defined in \eqref{2_4} and \eqref{2_14} respectively. Then a consistent discretization method is also convergent and the order of the convergence is the same order as the consistency.
\end{Theorem}
\begin{proof}
	Let $\boldsymbol{E}(t)$ be the global discretization error and is defined by $\boldsymbol{E}(t)=\boldsymbol{N}(t)-\boldsymbol{\hat{N}}(t)$ for all $\boldsymbol{N},\boldsymbol{\hat{N}}\in\mathbb{R}^I$. Now, we can write
	\begin{align}\label{3_32}
		\frac{\dd \boldsymbol{E}(t)}{\dd t}=\frac{\dd \boldsymbol{N}(t)}{\dd t}-\frac{\dd \boldsymbol{\hat{N}}(t)}{\dd t}=\boldsymbol{\hat{J}(\boldsymbol{N}(t))}-\boldsymbol{\hat{J}(\boldsymbol{\hat{N}}(t))}.
	\end{align}
	From Proposition \ref{pros_3_1},  $\boldsymbol{\hat{J}}$  satisfies  Lipschitz condition, so there exist Lipschitz constant $\gamma<\infty$ such that
	\begin{align}\label{3_33}
		||\boldsymbol{\hat{J}}(\boldsymbol{N})-\boldsymbol{\hat{J}}(\boldsymbol{\hat{N}})||\le \gamma||\boldsymbol{N}-\boldsymbol{\hat{N}}||\hspace{0.2cm}\text{ for all}\hspace{0.2cm}\boldsymbol{N},\boldsymbol{\hat{N}}\in\mathbb{R}^I.
	\end{align}
	Using  Lipschitz condition \eqref{3_33} in \eqref{3_32}, we obtain
	\begin{align}
		\frac{\dd|| \boldsymbol{E}(t)||}{\dd t}\leq||\boldsymbol{\hat{J}(\boldsymbol{N}(t))}-\boldsymbol{\hat{J}(\boldsymbol{\hat{N}}(t))}||\leq\gamma||\boldsymbol{N}-\boldsymbol{\hat{N}}||=\gamma|| \boldsymbol{E}(t)||.
	\end{align}
	The above inequality implies that
	\begin{align}
		||\boldsymbol{E}(t)||\leq||\boldsymbol{E}(0)||\exp{(\gamma t)}.
	\end{align}
	When $\Delta x\to 0$,  with the relation $||\boldsymbol{E}(0)||=0$ and by the definition of global dicretization error, we deduce that $||\boldsymbol{E}(t)||=0$ that is  if the  new discretization method is consistent, then it is convergent  and the order of the convergence is same as the order of consistency.
\end{proof}
Thus, from the convergence theorem \ref{theo_2_2}, we can say that the numerical VAM \eqref{2_14} is convergent and the order of convergence of the VAM is same as the order of consistency.

\section{Two-dimensional collisional breakage equation}\label{sec_4}
In particulate process, a particle is characterized by different properties such as volume,  energy, moisture content etc. These properties need to be defined by more than one-dimension. Here, we assume two particle properties $x_1$ and $x_2$. The corresponding  collisional  nonlinear breakage equation in two particle properties (or dimensions)is written as 
\begin{align}\label{4_1}
	\frac{\partial n(x_1,x_2,t)}{\partial t}=&\int_{0}^{\infty}\int_{0}^{\infty}\int_{x_1}^{\infty} \int_{x_2}^{\infty} \beta(x_1,x_2|y_1,y_2;z_1,z_2) \K(y_1,y_2,z_1,z_2) n(y_1,y_2,t)n(z_1,z_2,t) \dd y_2 \dd y_1\dd z_2 \dd z_1\notag\\
	&- n(x_1,x_2,t)\int_{0}^{\infty} \int_{0}^{\infty}\K(x_1,x_2,y_1,y_2) n(y_1,y_2,t) \dd y_2 \dd y_1,
\end{align}
with the initial condition
\begin{equation}\label{4_2}
	n(x_1,x_2,0)=n_{0}(x_1,x_2)\geq0,\quad\text{for}\quad (x_1,x_2)\in\mathbb{R}^+\times\mathbb{R}^+.
\end{equation}
Here, all the functions present in equation \eqref{4_1} are defined similar to the one-dimensional collisional nonlinear  breakage equation.
The moments of the particles with properties $x_1$ and $x_2$ is given by
\begin{align}
	\mathcal{M}_{r_1,r_2}(t)=\int_{0}^{\infty}\int_{0}^{\infty}x_1^{r_1}x_2^{r_2}n(x_1,x_2,t)\dd x_2\dd x_1, \quad\text{for all}\quad r_1,r_2=0,1,2,\dots.
\end{align}
Here, the zeroth order moment $\mathcal{M}_{0,0}(t)$ represents the total number of particles, whereas the moments $\mathcal{M}_{1,0}(t)$ and $\mathcal{M}_{0,1}(t)$ denotes the total amount of property with respect to $x_1$ and $x_2$, respectively and the moment $\mathcal{M}_{1,1}(t)$ describes the hypervolume of the particles.\\
We truncate the model in the finite computational domain $[0,x_{\max}]^2$ with  $x_{\max}<\infty$ and the time is taken as  $[0,T]$. Next we discretize the computational domain  as cartesian grid into $I_1 I_2$ cells as $\Lambda_{i,j}:= [x_{1,i-1/2},x_{1,i+1/2}]\times[x_{2,j-1/2},x_{2,j+1/2}]$ with $1\le i\le I_1$ and $1\le j\le I_2$. Moreover,
$$x_{1,1/2}:=0, \quad x_{2,1/2}:=0,\quad  x_{1,I_1+1/2}:=R_1 \quad\text{and}\quad x_{2,I_2+1/2}:=R_2.$$
\noindent The representative of the particle population in the cell $(i,j)$ is denoted  by $(x_{1,i},x_{2,j})$ where
$$ x_{1,i}:=\frac{x_{1,i-1/2}+x_{1,i+1/2}}{2}\quad\text{and}\quad x_{2,j}:=\frac{x_{2,j-1/2}+x_{2,j+1/2}}{2},$$
and also
$$\Delta x_{1,i}:=x_{1,i+1/2}-x_{1,i-1/2}\quad\text{and}\quad\Delta x_{2,j}:=x_{2,j+1/2}-x_{2,j-1/2}.$$
Let $N_{i,j}(t)$ be the discrete  number density in $\Lambda_{i,j}$ at time $t$ is defined by 
\begin{align}\label{4_4}
	N_{i,j}(t)=\int_{\Lambda_{i,j}} n(x_1,x_2,t)\dd x_2\dd x_1.
\end{align} 
Integrating equation truncated equation \eqref{4_1} with respect to the volume variables $x_{1,i}$ and $x_{2,j}$ over each cell $\Lambda_{i,j}$, we obtain
\begin{equation}\label{4_5}
	\frac{\dd N_{i,j}(t)}{\dd t} =  B_{i,j}(t) - D_{i,j}(t).
\end{equation}
Here, $B_{i,j}(t)$ is the birth rate  and $D_{i,j}(t)$ is the death rate of particles and are given by
\begin{subequations}\label{4_6}
	\begin{align}
		B_{i,j}(t):=&\int_{\Lambda_{i,j}}\int_{0}^{R_1}\int_{0}^{R_2}\int_{x_1}^{R_1} \int_{x_2}^{R_2} \beta(x_1,x_2|y_1,y_2;z_1,z_2) \K(y_1,y_2,z_1,z_2) \notag\\&\hspace{3.8cm}\times n(y_1,y_2,t)n(z_1,z_2,t) \dd y_2 \dd y_1\dd z_2 \dd z_1\dd x_2\dd x_1,
	\end{align}
	and
	\begin{align}
		D_{i,j}(t):=&\int_{\Lambda_{i,j}}\int_{0}^{R_1}\int_{0}^{R_2} \K(x_1,x_2,y_1,y_2)n(x_1,x_2,t)n(y_1,y_2,t) \dd y_2 \dd y_1\dd x_2\dd x_1.
	\end{align}
\end{subequations}
Putting $\ds n(x_1,x_2,t)\approx \sum_{i=1}^{I_{1}}\sum_{i=1}^{I_{2}}N_{i,j}(t) \delta(x-x_{1,i})\delta(x-x_{2,j})$ in the above equations \eqref{4_5}-\eqref{4_6}, we get the discrete formulation of  two-dimensional collisional nonlinear breakage equation as
\begin{align}\label{4_7}
	\frac{\dd {N}_{i,j}(t)}{\dd t} = \hat{ B}_{i,j}(t) - \hat{D}_{i,j}(t),
\end{align}
where $\hat{N}_i(t)$ discrete approximation of the density function at time $t\geq0$.
Here, birth terms and death terms  are defined by
\begin{subequations}\label{4_8}
	\begin{align}
		\hat{B}_{i,j}(t):=&\sum_{p_1=1}^{I_{1}} \sum_{p_2=1}^{I_{2}}\sum_{q_1=i}^{I_{1}}\sum_{q_2=j}^{I_{2}} \beta_{i,j}^{q_1,q_2} \K(x_{1,q_1},x_{2,q_2},x_{1,p_1},x_{2,p_2}) {N}_{q_1,q_2}(t) {N}_{p_1,p_2}(t),\\
		\hat{D}_{i,j}(t):=& \sum_{p_1=1}^{I_{1}}\sum_{p_2=1}^{I_{2}} \K(x_{1,i},x_{2,j},x_{1,p_1},x_{2,p_2}){ N}_{i,j}(t)  {N}_{p_1,p_2}(t),
	\end{align}
\end{subequations}
where \; $\ds \beta_{i,j}^{q_1,q_2}= \int_{x_{1,i-1/2}}^{p_{1,i}^{q_1}}\int_{x_{2,j-1/2}}^{p_{2,j}^{q_2}}  \beta({x}_1,{x}_2|x_{1,q_1},x_{2,q_2};x_{1,p_1},x_{2,p_2})\dd {x}_2\dd {x}_1$\; and for all $k=1,2$ with $ l=i,j$,\;
$\ds p_{k,l}^{q_k}=
\begin{cases} 
	x_{k,l}, & \quad \mbox{if}\quad q_k=l,\\
	x_{k,l+1/2}, & \quad \mbox{otherwise}.
\end{cases}$\\
The total mass flux of the property for $x_1$ and $x_2$, respectively over $\Lambda_{i,j}$ are given by
\begin{align}\label{4_9}
	V_{{x_1,i}}(t):=&\int_{\Lambda_{i,j}}\int_{0}^{R_1}\int_{0}^{R_2}\int_{x_1}^{R_1} \int_{x_2}^{R_2}x_{1} \beta(x_1,x_2|y_1,y_2;z_1,z_2) \K(y_1,y_2,z_1,z_2)\notag\\&\hspace{3.8cm}\times n(y_1,y_2,t)n(z_1,z_2,t) \dd y_2 \dd y_1\dd z_2 \dd z_1\dd x_2\dd x_1,
\end{align}
and
\begin{align}
	V_{x_2,i}(t):=&\int_{\Lambda_{i,j}}\int_{0}^{R_1}\int_{0}^{R_2}\int_{x_1}^{R_1} \int_{x_2}^{R_2}x_{2} \beta(x_1,x_2|y_1,y_2;z_1,z_2) \K(y_1,y_2,z_1,z_2)\notag\\&\hspace{3.8cm}\times n(y_1,y_2,t)n(z_1,z_2,t) \dd y_2 \dd y_1\dd z_2 \dd z_1\dd x_2\dd x_1.
\end{align}
The	discrete  mass flux  of property for  $x_1$ and $x_2$, respectively  over $\Lambda_{i,j}$  are obtained as 
\begin{align}\label{4_11}
	\hat{V}_{x_1,i}(t):=&\sum_{p_1=1}^{I_{1}} \sum_{p_2=1}^{I_{2}}\sum_{q_1=i}^{I_{1}}\sum_{q_2=j}^{I_{2}}  \K(x_{1,q_1},x_{2,q_2},x_{1,p_1},x_{2,p_2}) {N}_{q_1,q_2}(t) {N}_{p_1,p_2}(t)\notag\\&\hspace{3cm}\times\int_{x_{1,i-1/2}}^{p_{1,i}^{q_1}}\int_{x_{2,j-1/2}}^{p_{2,j}^{q_2}}x_1 \beta({x}_1,{x}_2|x_{1,q_1},x_{2,q_2};x_{1,p_1},x_{2,p_2})\dd {x}_2\dd {x}_1,
\end{align}
and
\begin{align}\label{4_12}
	\hat{V}_{x_2,j}(t):=&\sum_{p_1=1}^{I_{1}} \sum_{p_2=1}^{I_{2}}\sum_{q_1=i}^{I_{1}}\sum_{q_2=j}^{I_{2}}  \K(x_{1,q_1},x_{2,q_2},x_{1,p_1},x_{2,p_2}){N}_{q_1,q_2}(t) {N}_{p_1,p_2}(t)\notag\\&\hspace{3cm}\times\int_{x_{1,i-1/2}}^{p_{1,i}^{q_1}}\int_{x_{2,j-1/2}}^{p_{2,j}^{q_2}}x_2 \beta({x}_1,{x}_2|x_{1,q_1},x_{2,q_2};x_{1,p_1},x_{2,p_2})\dd {x}_2\dd {x}_1.
\end{align}	
The average property values of all new born particles for property with respect to  $x_1$ and $x_2$, respectively in $\Lambda_{i,j}$ are given by
\begin{align}
	\bar{v}_{{1,i}}:=\frac{\hat{V}_{x_1,i}}{\hat{B}_{i,j}}\quad\text{and}\quad
	\bar{v}_{{2,j}}:=\frac{\hat{V}_{x_2,j}}{\hat{B}_{i,j}}.
\end{align}	
Similar argument is used as one-dimensional case to formulate the numerical scheme VAM.\\
Here, $\boldsymbol{\hat{N}}=\left\{\hat{N}_{i,j}\right\}\in\mathbb{R}^{I_1}\times\mathbb{R}^{I_2}$ be nonnegative vector for $i=1,2,\dots,I_1$ and $j=1,2,\dots,I_2$.
\subsection{Cell volume average method in two-dimension}
The numerical scheme for two-dimensional collisional nonlinear  breakage equation by VAM is defined by
\begin{align}
	\frac{\dd \hat{N}_{i,j}(t)}{\dd t}=&\sum_{p=0}^{1}\sum_{q=0}^{1}\lambda_{i,j}^{-,-}(\bar{v}_{{1,i-p}},\bar{v}_{{2,j-q}})H[(-1)^p({x}_{1,i-p}-\bar{v}_{1,i-p})]H[(-1)^q({x}_{2,j-q}-\bar{v}_{2,j-q})] \hat{B}_{i-p,j-q}(t)\notag\\&+\sum_{p=0}^{1}\sum_{q=0}^{1}\lambda_{i,j}^{-,+}(\bar{v}_{{1,i-p}},\bar{v}_{{2,j+q}})H[(-1)^p({x}_{1,i-p}-\bar{v}_{1,i-p})]H[(-1)^q({x}_{2,j+q}-\bar{v}_{2,j+q})] \hat{B}_{i-p,j+q}(t)\notag\\&+
	\sum_{p=0}^{1}\sum_{q=0}^{1}\lambda_{i,j}^{+,-}(\bar{v}_{{1,i+p}},\bar{v}_{{2,j-q}})H[(-1)^p({x}_{1,i+p}-\bar{v}_{1,i+p})]H[(-1)^q({x}_{2,j-q}-\bar{v}_{2,j-q})] \hat{B}_{i+p,j-q}(t)\notag\\&+\sum_{p=0}^{1}\sum_{q=0}^{1}\lambda_{i,j}^{+,+}(\bar{v}_{{1,i+p}},\bar{v}_{{2,j+q}})H[(-1)^p({x}_{1,i+p}-\bar{v}_{1,i+p})]H[(-1)^q({x}_{2,j+q}-\bar{v}_{2,j+q})] \hat{B}_{i+p,j+q}(t)\notag\\&-\sum_{p=1}^{I_{x_1}}\sum_{q=1}^{I_{x_2}} \K(x_{1,i},x_{1,p},x_{2,j},x_{2,q})\hat{ N}_{i,j}  \hat{N}_{p,q},
\end{align}
where $\hat{B}_{i,j}(t)$ defined by \eqref{4_8} and 	$\ds \lambda_{i,j}^{\pm,\pm}(\bar{v}_{1,i},\bar{v}_{2,j}):=\frac{\ds (\bar{v}_{1,i}-x_{1,i\pm 1})(\bar{v}_{2,j}-x_{2,j\pm 1})}{\ds (x_{1,i}-x_{1,i \pm 1})(x_{2,j}-x_{2,j\pm 1})}$
and $H$ is the Heaviside function defined as in previous one-dimensional case.

\section{Numerical discussion}	\label{sec_5}
This section is devoted to examine the accuracy of  VAM against existing fixed pivot technique and exact results for particular initial conditions to solve the collisional nonlinear breakage equation. 
This section is divided in two sub parts based on dimensions of the collisional  nonlinear breakage equation. For one dimensional problem, two  and for multi-dimensional  PBEs, two  test cases are analyzed. To access the accuracy and versatility of VAM, the results are compared with the exact values and FPT based approximations. For the numerical validation both thematic and data analysis is carried out.

\subsection{One-dimensional case}
This part deals with the one-dimensional numerical results  with two test examples together with a detailed error estimations between the new scheme VAM and FPT. Additionally, we  show  comparison with exact solutions. 
The significant key finding is that the VAM performs much better than existing fixed pivot techniques \cite{MR4566091} over  random grids. 
The relative errors of the moments are calculated using the following relation
\begin{align}\label{5_1}
	E_{\mathcal{M}_r}(t) = \left|\frac{\mathcal{M}_r(t)-\hat{\mathcal{M}_r}(t)}{\mathcal{M}_r(t)}  \right|.
\end{align}
Here, $\mathcal{M}_r(t)$ and $\hat{\mathcal{M}_r}(t)$ are the exact and  numerical value of the $r^{th}$ moment, respectively. 

During numerical computations, dimensionless values are considered of all the concerned quantities which are obtained by dividing the quantities with  corresponding initial values. The computational domain for particle size is taken as $D := [10^{-9}, 1]$ and it is discretized into $30$ nonuniform sub-intervals using the geometric recurrence relation $x_{i+1/2} = rx_{i-1/2}$, where $r=1.4$. All simulations are performed in a HP Z6 G4 workstation and using MATLAB R2023b software.

\begin{example}\label{ex1}
	Consider  collisional nonlinear breakage equation with monodispresed  initial condition $\ds n(x,0)= \begin{cases}
		1,&\mbox{x=1}\\
		0, & \mbox{otherwise}
	\end{cases}$ with  collisional kernel $ \ds \K(x,y)=xy$ and breakage function as $\ds \beta(x|y;z)=\frac{2}{y}$. This breakage function splits the mother particle into two daughter particles.
\end{example}

Figure \ref{B_f1} is the particle number density function representation at time $t=1$ in log scale for three different type of grids. For nonuniform grid, FPT and VAM  both predict number density with high accuracy and are inline the exact one as shown in Figure \ref{B_f1_1}.  However, when Locally-uniform grids are considered, FPT over-predicted number density at few points and VAM generated accurate results (refer to Figure \ref{B_f1_4}). Moreover, from Figure \ref{B_f1_5}for random grid type FPT seems to fail in approximating the density function. On the other hand, VAM produced satisfactory results. To summarize, we can say that irrespective of grid type, VAM gives highly efficient solution.

\begin{figure}[H]
	\begin{subfigure}{.325\textwidth}
		\centering
		\includegraphics[width=1.00\textwidth]{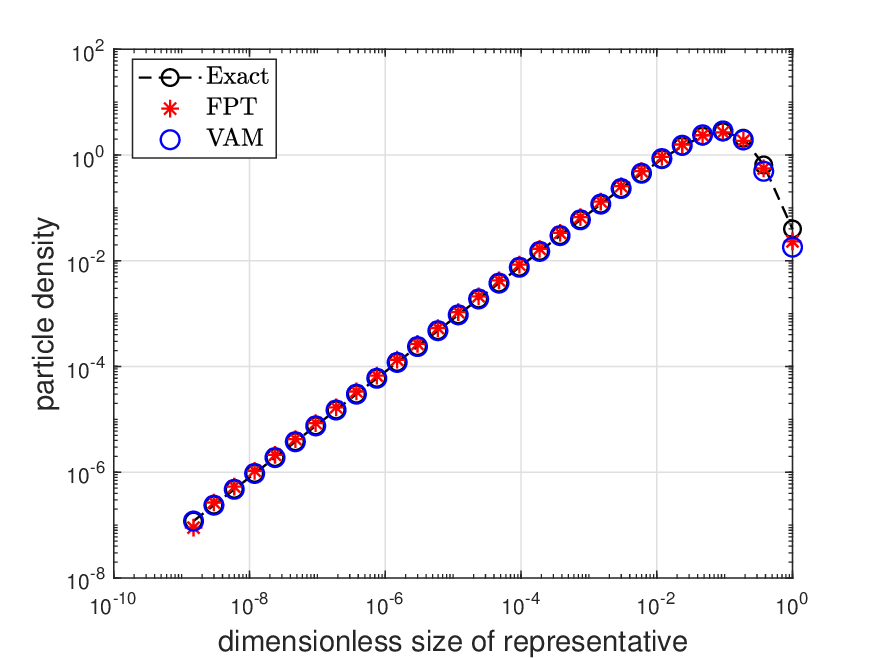}
		\caption{Nonuniform grids}
		\label{B_f1_1}
	\end{subfigure}
	\begin{subfigure}{.325\textwidth}
		\centering
		\includegraphics[width=1.00\textwidth]{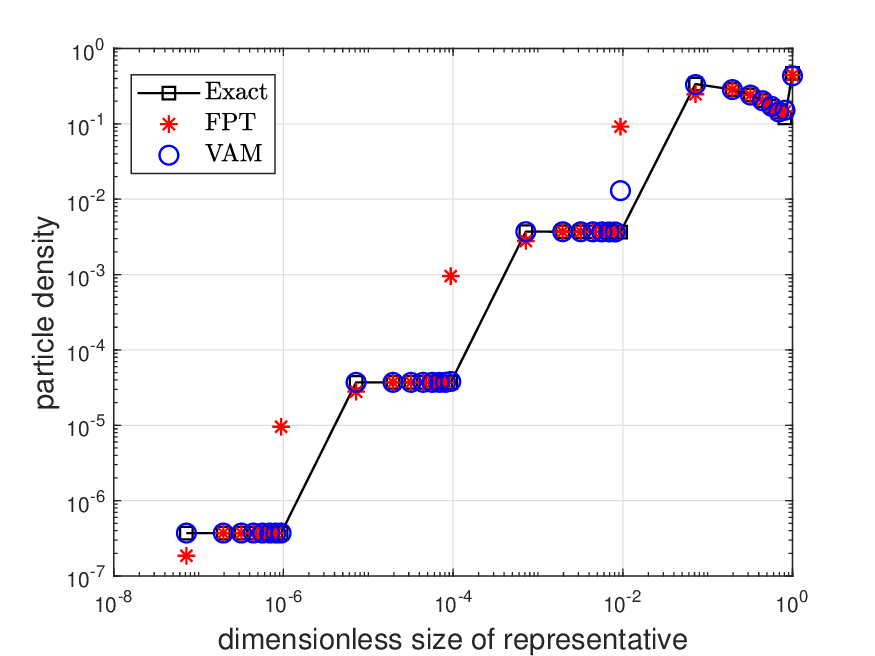}
		\caption{ Locally-uniform grids}
		\label{B_f1_4}
	\end{subfigure}
	\begin{subfigure}{.325\textwidth}
		\centering
		\includegraphics[width=1.00\textwidth]{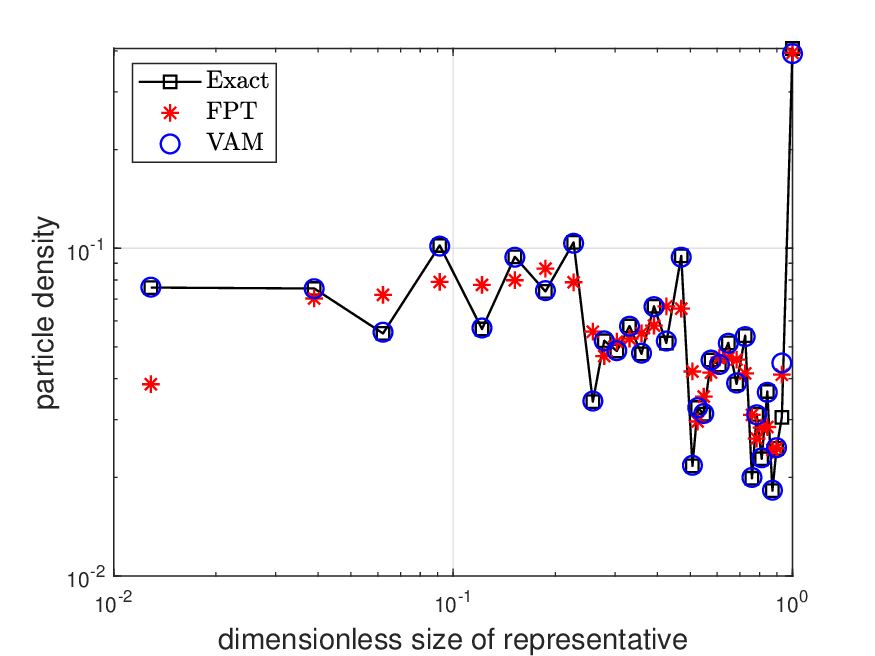}
		\caption{Random grids}
		\label{B_f1_5}
	\end{subfigure}
	\caption{Comparison of number density for different grid types at $t=1$ for test case \ref{ex1}.  }
	\label{B_f1}
\end{figure}

Different order moments are calculated over nonuniform grids. Since, over locally-uniform and random grids, FPT fails to predict number density, therefore they are not considered for moments calculation. The zeroth and first order moments are depicted in Figure \ref{B_f1_2} and both are inline with the exact and FPT based moments. Table \ref{t_1_2}  supports the insignificant error claim. Second and third moment is plotted in Figure \ref{B_f1_3} and \ref{B_f1_31}. Moment error data is presented in Table \ref{t_1_2}. As compared to FPT, the new scheme VAM produces more accurate results.
\begin{figure}[H]
	\centering
	\begin{subfigure}{.325\textwidth}
		\centering
		\includegraphics[width=1.00\textwidth]{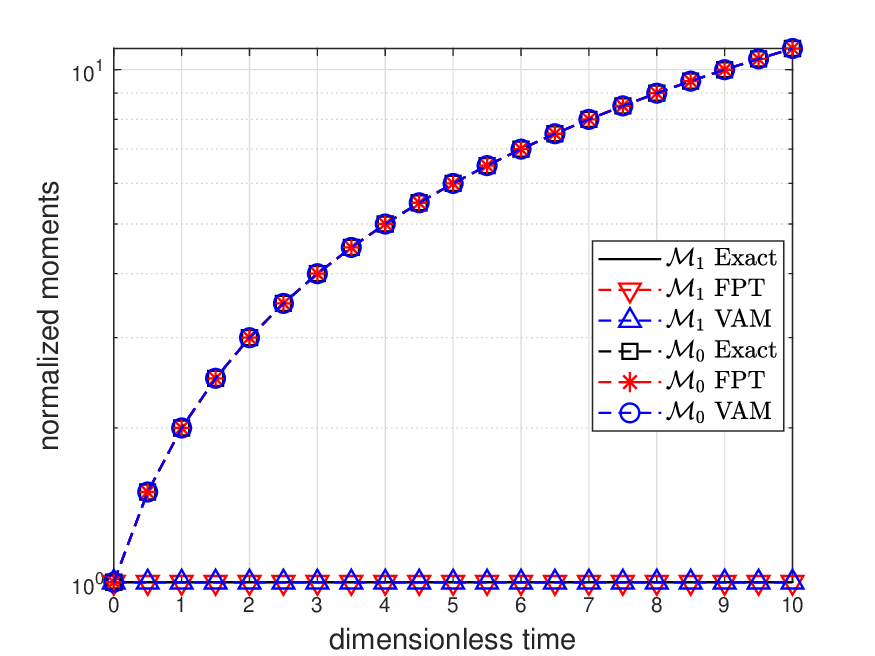}
		\caption{Zeroth and first moment}
		\label{B_f1_2}
	\end{subfigure}
	\begin{subfigure}{.325\textwidth}
		\centering
		\includegraphics[width=1.00\textwidth]{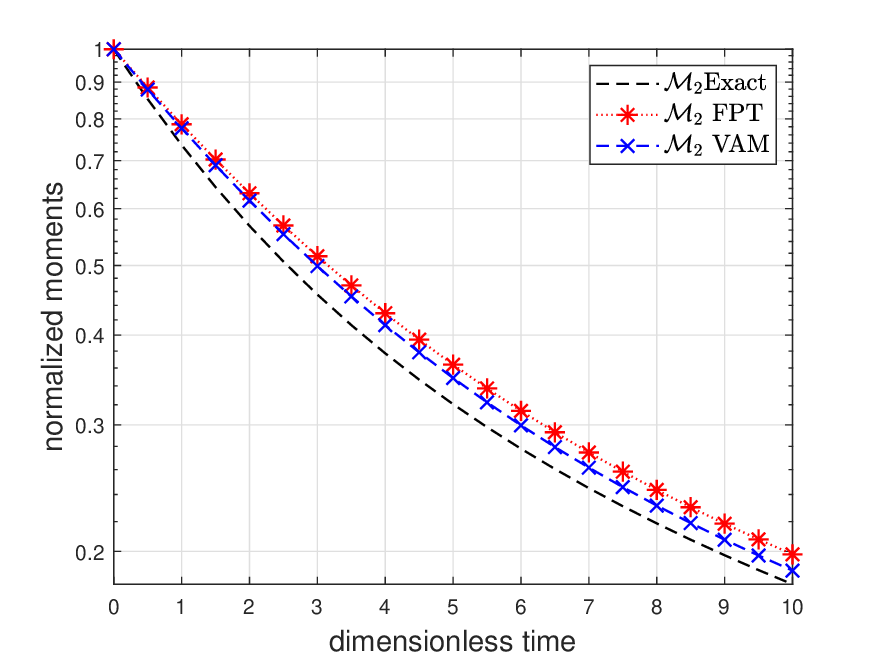}
		\caption{Second moment}
		\label{B_f1_3}
	\end{subfigure}
	\begin{subfigure}{.325\textwidth}
		\centering
		\includegraphics[width=1.00\textwidth]{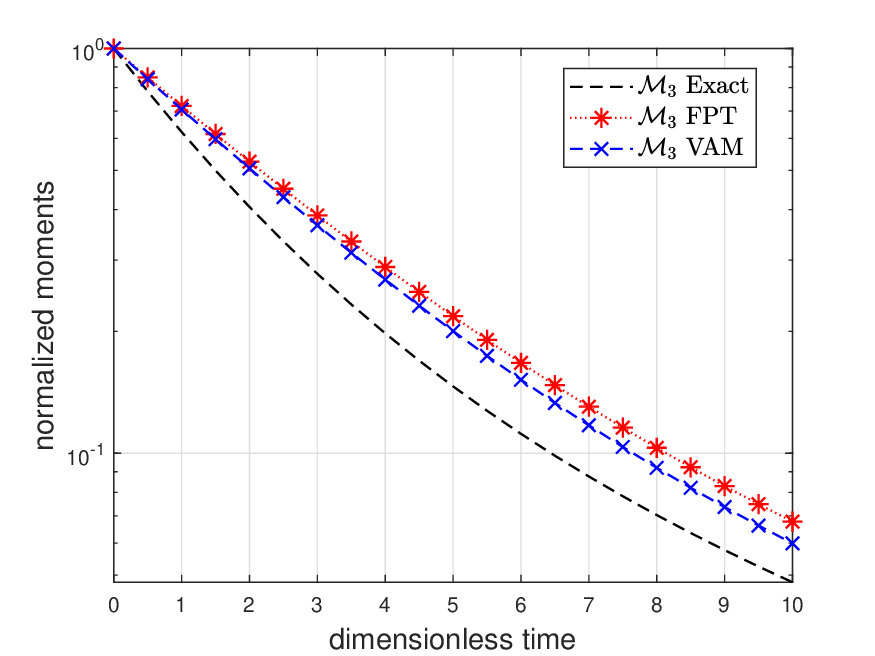}
		\caption{Third moment}
		\label{B_f1_31}
	\end{subfigure}
	\caption{Different order moments for test case \ref{ex1}. }
	\label{B1_f11}
\end{figure}

\begin{table}[htbp!]
	\centering
	\begin{tabular}{|c|  c |c |c|c|c|c| }     
		\hline
		Method & \multicolumn{3}{c| }{VAM} & \multicolumn{3}{c|}{FPT}\\ [0.5ex]
		\hline
		Time & $E_{\mathcal{M}_0}$ & $E_{\mathcal{M}_1}$ & $E_{\mathcal{M}_2}$ & $E_{\mathcal{M}_0}$& $E_{\mathcal{M}_1}$ & $E_{\mathcal{M}_2}$ \\ [0.5ex]
		\hline 
		2 & 1.43$\times 10^{-9}$ &  8.23$\times 10^{-9}$ & 8.46$\times 10^{-2}$ & 2.76$\times 10^{-9}$ & 8.23$\times 10^{-9}$  & 1.12$\times 10^{-1}$\\ [0.5ex]
		\hline 
		4 &  5.56$\times 10^{-9}$ & 8.29$\times 10^{-8}$  & 1.37$\times 10^{-1}$  & 3.17$\times 10^{-9}$ & 8.29$\times 10^{-8}$ & 9.46$\times 10^{-2}$\\ [0.5ex]
		\hline 
		6 &  5.91$\times 10^{-9}$ & 4.26$\times 10^{-8}$  & 1.28$\times 10^{-1}$ & 2.49$\times 10^{-9}$ & 4.26$\times 10^{-8}$ & 7.78$\times 10^{-2}$\\ [0.5ex]
		\hline 
		8 &  3.02$\times 10^{-9}$ & 7.62$\times 10^{-8}$ & 5.88$\times 10^{-2}$& 7.44$\times 10^{-9}$ & 7.62$\times 10^{-8}$  & 1.13$\times 10^{-1}$ \\ [0.5ex]
		\hline 
		10  & 3.53$\times 10^{-9}$ &  9.29$\times 10^{-8}$  & 4.49$\times 10^{-2}$ & 8.96$\times 10^{-9}$ & 9.29$\times 10^{-8}$  & 1.01$\times 10^{-1}$\\ [0.5ex]
		\hline 
		
	\end{tabular}
	\caption{Error analysis of different order moments for test case \ref{ex1}.}
	\label{t_1_2}
\end{table}

\begin{example}\label{ex2}
	Consider  collisional nonlinear breakage equation with monodispresed  initial condition  with  collisional kernel $ \ds \K(x,y)=1$ and breakage function as $\ds \beta(x|y;z)=\frac{4x^2}{y^3}$. 
\end{example}

Particle number density in log scale at time $t=1$ is presented in the Figure \ref{B_f2} for three type of grids. For small particle sizes, FPT under-predicts it whereas VAM gives accurate approximation for nonuniform grids. For locally-uniform and random grids, FPT fails predicts the density function as shown in Figure \ref{B_f2_4}  and \ref{B_f2_5}, respectively. However, VAM produces more accurate results for both grid types.

\begin{figure}[H]
	\begin{subfigure}{.325\textwidth}
		\centering
		\includegraphics[width=1.00\textwidth]{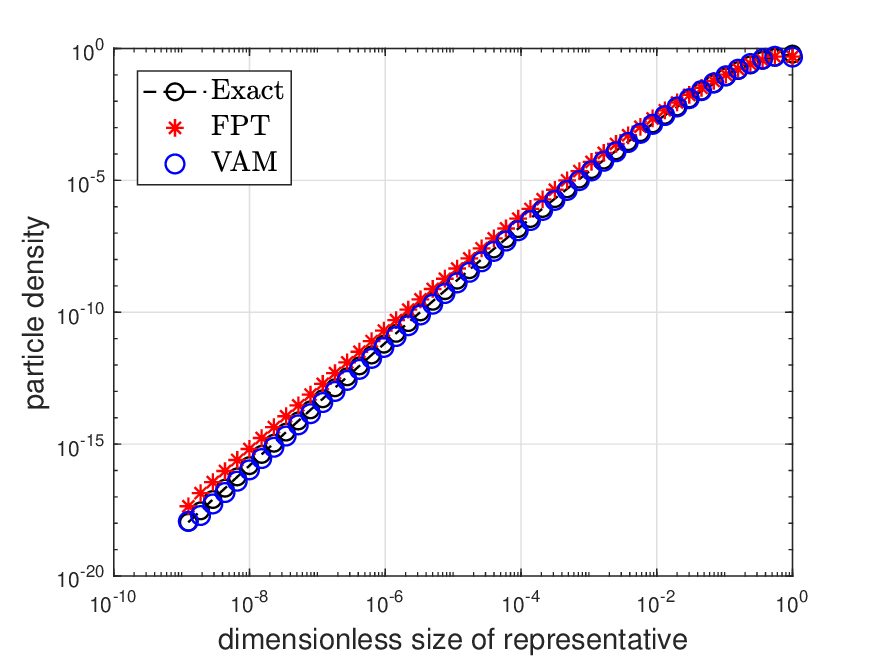}
		\caption{Nonuniform grids}
		\label{B_f2_1}
	\end{subfigure}
	\begin{subfigure}{.325\textwidth}
		\centering
		\includegraphics[width=1.00\textwidth]{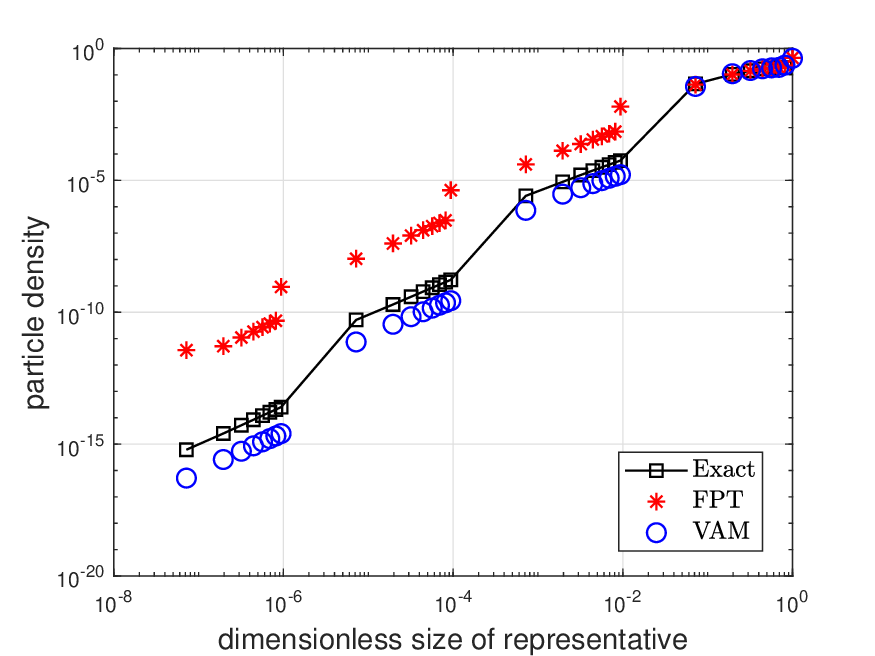}
		\caption{Locally-uniform grids}
		\label{B_f2_4}
	\end{subfigure}
	\begin{subfigure}{.325\textwidth}
		\centering
		\includegraphics[width=1.00\textwidth]{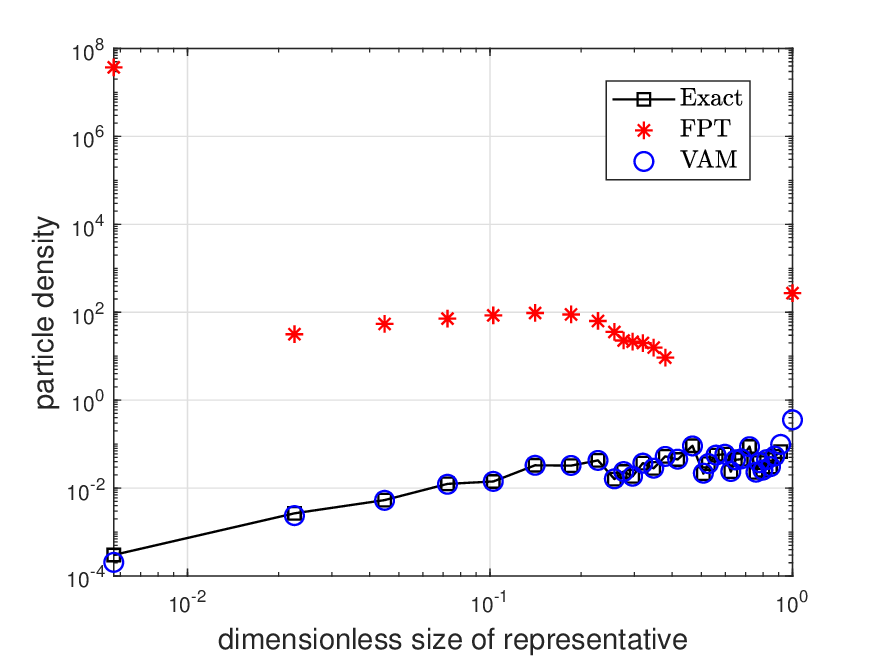}
		\caption{Random grids}
		\label{B_f2_5}
	\end{subfigure}
	\caption{Comparison of number density for different grid types at $t=1$ for test case \ref{ex2}. }
	\label{B_f2} 
\end{figure}
The first two moments are predicted with high accuracy by both the methods (refer to Figure \ref{B_f2_2} and Table \ref{t_2_2}) for nonuniform grids. For second order moment VAM generates less error and gives better prediction as compared to FPT. A similar observation can be made from the Table \ref{t_2_2}.

\begin{figure}[H]
	\centering
	\begin{subfigure}{.325\textwidth}
		\centering
		\includegraphics[width=1.00\textwidth]{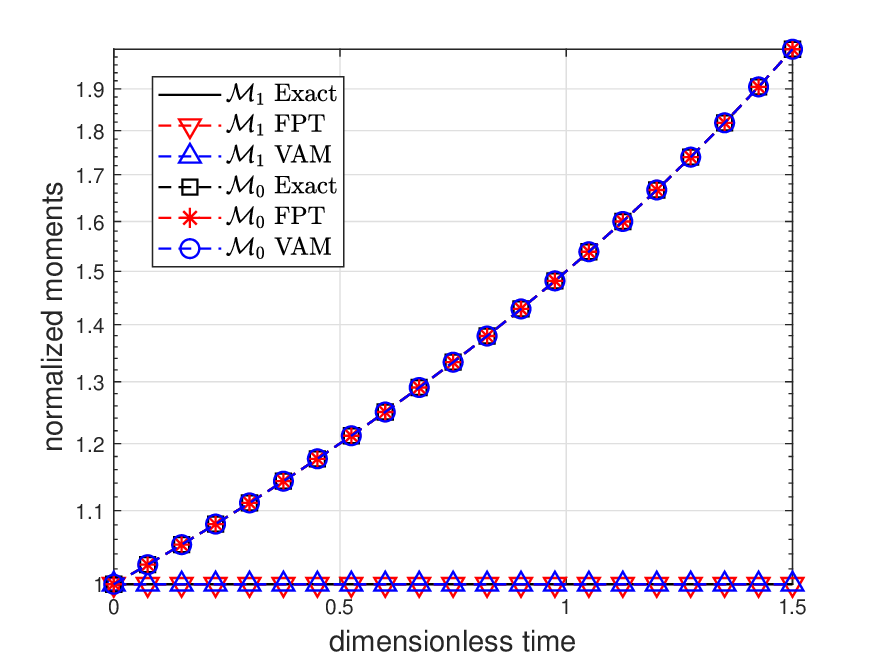}
		\caption{Zeroth and first moment}
		\label{B_f2_2}
	\end{subfigure}
	\begin{subfigure}{.325\textwidth}
		\centering
		\includegraphics[width=1.00\textwidth]{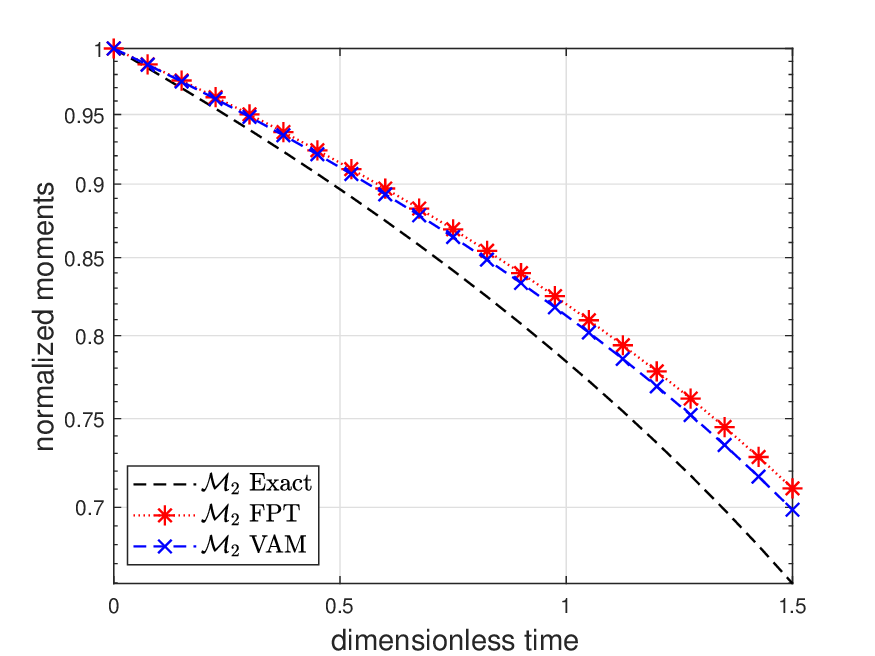}
		\caption{Second moment}
		\label{B_f2_3}
	\end{subfigure}
	\begin{subfigure}{.325\textwidth}
		\centering
		\includegraphics[width=1.00\textwidth]{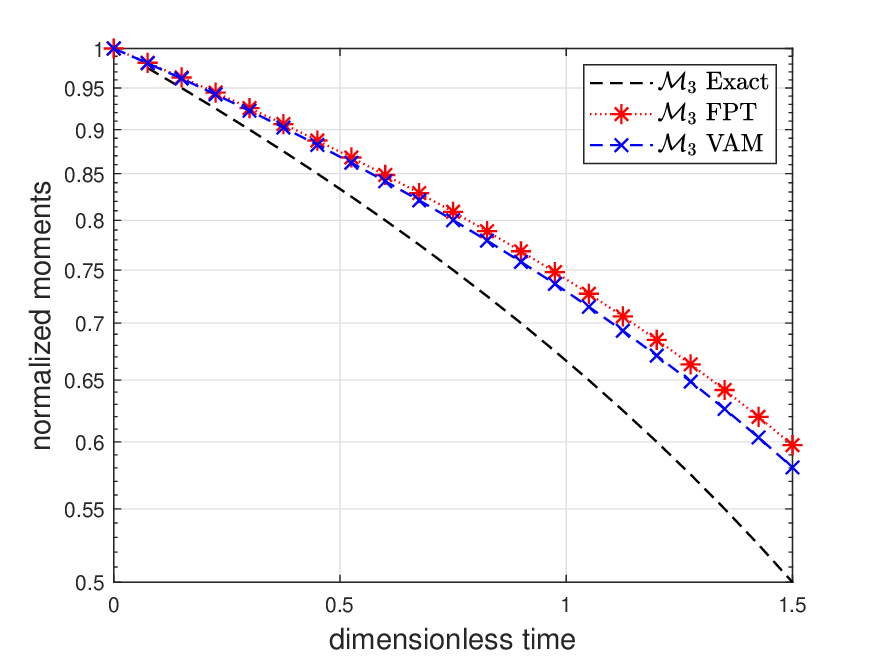}
		\caption{Third moment}
		\label{B_f2_31}
	\end{subfigure}
	\caption{Different order moments for test case \ref{ex2}. }
	\label{B1_f122}
\end{figure}
\begin{table}[htbp!]
	\centering
	\begin{tabular}{|c|  c |c |c|c|c|c| }     
		\hline
		Method & \multicolumn{3}{c| }{VAM} & \multicolumn{3}{c|}{FPT}\\ [0.5ex]
		\hline
		Time & $E_{\mathcal{M}_0}$ & $E_{\mathcal{M}_1}$ & $E_{\mathcal{M}_2}$ & $E_{\mathcal{M}_0}$ & $E_{\mathcal{M}_1}$ & $E_{\mathcal{M}_2}$  \\ [0.5ex]
		\hline 
		0.3 & 2.57$\times 10^{-11}$ & 1.45$\times 10^{-11}$  & 1.21$\times 10^{-2}$ & 2.57$\times 10^{-11}$ &  1.45$\times 10^{-11}$ & 1.01$\times 10^{-2}$\\ [0.5ex]
		\hline 
		0.6 & 4.94$\times 10^{-11}$ &  1.42$\times 10^{-11}$ & 2.07$\times 10^{-2}$& 4.94$\times 10^{-11}$ & 1.42$\times 10^{-11}$  & 2.53$\times 10^{-2}$ \\ [0.5ex]
		\hline 
		0.9 & 4.63$\times 10^{-11}$ &  6.63$\times 10^{-11}$ & 3.22$\times 10^{-2}$& 4.63$\times 10^{-11}$ & 6.63$\times 10^{-11}$  & 4.02$\times 10^{-2}$ \\ [0.5ex]
		\hline 
		1.2  & 3.47$\times 10^{-10}$ &  2.26$\times 10^{-10}$ & 4.48$\times 10^{-2}$& 3.47$\times 10^{-10}$ & 2.26$\times 10^{-10}$  & 5.71$\times 10^{-2}$\\ [0.5ex]
		\hline 
		1.5 & 7.72$\times 10^{-10}$  &   3.19$\times 10^{-10}$  & 5.90$\times 10^{-2}$& 7.72$\times 10^{-10}$ & 3.19$\times 10^{-10}$ & 7.67$\times 10^{-2}$ \\ [0.5ex]
		\hline 
		
	\end{tabular}
	\caption{Error analysis of different order moments for test case \ref{ex2}.}
	\label{t_2_2}
\end{table}	
\subsection{Experimental order of convergence}

The theoretical results demonstrating the order of convergence of  the new scheme VAM are examined against  the numerical estimated order of convergence (EOC) for analytically tractable kernels \cite{MR4566091} using the following formula:
\begin{align}
	\text{EOC}=\ds \frac{\ds \ln\left(\ds\frac{E_I}{E_{2I}}\right)}{\ds \ln(2)},
\end{align}
where $E_I$ is the $L^1$ error norm calculated by 
\begin{align}
	E_I:=\sum_{j=1}^{I}|N_j-\hat{N}_j|=||N-\hat{N}||\quad\text{and}\quad L^1_{\text{error}}:=\ds \frac{||N-\hat{N}||}{||N||}=\ds \frac{E_I}{||N||}.
\end{align}
Here, the subscript $I$ corresponds to the degrees of freedom. 
The EOC is calculated for all the test cases considered  for uniform, nonuniform, locally-uniform and random grids (see Table \ref{t1_EOC} and \ref{t2_EOC}). For calculating the EOC, computation size domain $[10^{-9},1]$ is considered,  which is discretized into 30 grids initially. These grids are doubled in each iteration and five iterations are performed.  EOC is calculated over four different grid types, namely uniform,  nonuniform, locally-uniform and random.  For all simulations end time considered as $t=1$.

\begin{table}[H]
	\centering
	\caption{$L^1$ error and EOC for test case \ref{ex1} for VAM.}
	\label{t1_EOC}	
	\begin{tabular}{ |c |c c|c c|c c|c c|}
		\hline
		& \multicolumn{2}{c|}{Nonuniform grids} & \multicolumn{2}{c|}{Uniform grids} & \multicolumn{2}{c|}{Locally-uniform grids} & \multicolumn{2}{c|}{ Random grids} \\
		\hline
		Grids & $L^1$ error  & EOC & $L^1$ error  & EOC & $L^1$ error  & EOC  & $L^1$ error  & EOC \\
		\hline
		30 & 0.97$\times 10^{-1}$   & 0 & 3.56$\times 10^{-2}$        & 0     & 8.25$\times 10^{-1}$ & 0 & 3.59$\times 10^{-1}$ & 0  \\
		60 & 4.55$\times 10^{-1}$   & 1.16   & 1.71$\times 10^{-2}$ & 1.05 & 2.31$\times 10^{-1}$ & 1.18 & 2.28$\times 10^{-1}$ & 0.57\\
		120 & 2.03$\times 10^{-1}$   & 1.11& 8.36$\times 10^{-3}$   & 1.03 & 1.07$\times 10^{-1}$ & 1.10 & 1.53$\times 10^{-1}$ & 0.65\\
		240 & 9.34$\times 10^{-2}$   & 1.09 & 4.13$\times 10^{-3}$   & 1.01 & 4.91$\times 10^{-2}$ & 1.11 & 9.35$\times 10^{-2}$ & 0.72  \\
		480 & 4.45$\times 10^{-2}$   & 1.06 & 2.05$\times 10^{-3}$   & 1.01  & 2.31$\times 10^{-2} $ & 1.08 & 5.81$\times 10^{-2}$ & 0.87\\
		\hline
	\end{tabular}
\end{table}

\begin{table}[H]
	\centering
	\caption{$L^1$ error and EOC for test case \ref{ex2} for VAM.}
	\label{t2_EOC}	
	\begin{tabular}{ |c |c c|c c|c c|c c|}
		\hline
		& \multicolumn{2}{c|}{Nonuniform grids} & \multicolumn{2}{c|}{Uniform grids} & \multicolumn{2}{c|}{Locally-uniform grids} & \multicolumn{2}{c|}{ Random grids} \\
		\hline
		Grids & $L^1$ error  & EOC & $L^1$ error  & EOC & $L^1$ error  & EOC  & $L^1$ error  & EOC \\
		\hline
		30 & 6.21$\times 10^{-1}$   & 0 & 6.32$\times 10^{-2}$        & 0     & 6.66$\times 10^{-1}$ & 0 & 6.32$\times 10^{-2}$ & 0  \\
		60 & 3.97$\times 10^{-1}$   & 0.65   & 3.19$\times 10^{-2}$ & 1.00 & 2.42$\times 10^{-1}$ & 0.76 & 5.34$\times 10^{-2}$ & 0.67\\	
		120 & 2.24$\times 10^{-1}$   & 0.82 & 1.59$\times 10^{-2}$   & 1.00 & 1.42$\times 10^{-1}$ & 0.84 & 2.64$\times 10^{-2}$ & 0.76\\
		240 & 1.25$\times 10^{-1}$   & 0.84& 8.07$\times 10^{-3}$   & 0.99 & 7.91$\times 10^{-2}$ & 0.94 & 1.41$\times 10^{-2}$ & 0.85  \\
		480 & 6.74$\times 10^{-2}$   & 0.89 & 4.03$\times 10^{-3}$   & 0.99  & 4.12$\times 10^{-2} $ & 0.97 & 1.02$\times 10^{-2}$ & 0.89\\
		\hline
	\end{tabular}
\end{table}
Table \ref{t1_EOC} presents the $L^1$ error  and EOC data for different grid types for Example \ref{ex1}. It can be observed from the data that VAM is first order accurate irrespective of grid choice. For the first three type of grids, the convergence order is achieved rapidly, however for random grids, it is convergence is comparatively slow to other grids. It is interesting to note that for all the grids, $L^1$ error  is decreasing by half when the number of grids are doubled. For Example \ref{ex2}, Table \ref{t2_EOC} contains the data for  $L^1$ error  and EOC. Similar to previous case, first order of convergence is noticed for all four grid types.

Therefore, by numerical experiments we can say that the VAM performs accurately and predict the convergence results with better accuracy in comparison with existing FPT \cite{MR4566091}. 

\begin{Remark}\label{rem_4_1}
	Here, we verify the claim made in Remark \ref{rem_3_4}. Consider  $\beta(x|y;z)=\ds \frac{12x}{y^2}\left(1-\frac{x}{y}\right)$ and calculate  the EOC for VAM over uniform,  nonuniform, locally-uniform and random grids. Table \ref{t3_EOC} supports the claim.
	
	\begin{table}[H]
		\centering
		\caption{$L^1$ error and EOC for VAM with  $\beta(x|y;z)=\ds \frac{12x}{y^2}\left(1-\frac{x}{y}\right)$.}
		\label{t3_EOC}	
		\begin{tabular}{ |c |c c|c c|c c|c c|}
			\hline
			& \multicolumn{2}{c|}{Nonuniform grids} & \multicolumn{2}{c|}{Uniform grids} & \multicolumn{2}{c|}{Locally-uniform grids} & \multicolumn{2}{c|}{ Random grids} \\
			\hline
			Grids & $L^1$ error  & EOC & $L^1$ error  & EOC & $L^1$ error  & EOC  & $L^1$ error  & EOC \\
			\hline
			30 & 2.45$\times 10^{-2}$   & 0 & 1.06$\times 10^{-0}$        & 0     & 9.92$\times 10^{-1}$ & 0 & 2.41$\times 10^{-2}$ & 0  \\
			60 & 7.24$\times 10^{-3}$   & 1.77   & 4.14$\times 10^{-1}$ & 1.35 & 2.11$\times 10^{-1}$ & 2.23 & 1.89$\times 10^{-2}$ & 0.85\\	
			120 & 2.06$\times 10^{-3}$   & 1.85 & 1.33$\times 10^{-1}$   & 1.63 & 6.57$\times 10^{-2}$ & 1.68 & 8.39$\times 10^{-3}$ & 1.18\\
			240 & 5.76$\times 10^{-4}$   & 1.91& 3.79$\times 10^{-2}$   & 1.81 & 1.65$\times 10^{-2}$ & 1.99 & 4.25$\times 10^{-3}$ & 0.98  \\
			480 & 1.75$\times 10^{-4}$   & 1.94 & 1.01$\times 10^{-2}$   & 1.91  & 4.21$\times 10^{-3} $ & 1.98 & 2.87$\times 10^{-3}$ & 1.08\\
			\hline
		\end{tabular}
	\end{table}
\end{Remark}


\subsection{Two-dimensional case}

\begin{example}
	Consider multi-dimensional collisional  nonlinear breakage equation with monodispresed  initial condition with  collisional kernel $ \ds \K(x_1,x_2,y_1,y_2)=x_1x_2y_1y_2$ and breakage function as 
	$(i)$ $\ds \beta(x_1,x_2|y_1,y_2;z_1,z_2)=\frac{4}{y_1y_2}$ and
	$(ii)$ $\ds \beta(x_1,x_2|y_1,y_2;z_1,z_2)=\frac{2}{y_1y_2}$.
\end{example}
\noindent Case $(i)$: For this case, exact zeroth  $\mathcal{M}_{0,0}(t)$ and first moment $\mathcal{M}_{1,1}(t)$ can be calculated. However, the  total amount of property with respect to $x_1$ and $x_2$ or first cross moments $\mathcal{M}_{0,1}(t)$ and $\mathcal{M}_{1,0}(t)$ are unavailable. In Figure \ref{B_f3_1} zeroth moment and hypervolume are plotted and compared against the exact values. An excellent prediction of both the integral properties can be noticed. The total of first order mixed moments is showcased in Figure \ref{B_f3_2}. The prediction of average hypervolume depends on the zeroth and hypervolume as it is defined as their ratio. A high accurate approximation of averaged value is anticipated and shown in Figure \ref{B_f3_3}.

\begin{figure}[H]
	\begin{subfigure}{.325\textwidth}
		\centering
		\includegraphics[width=1.00\textwidth]{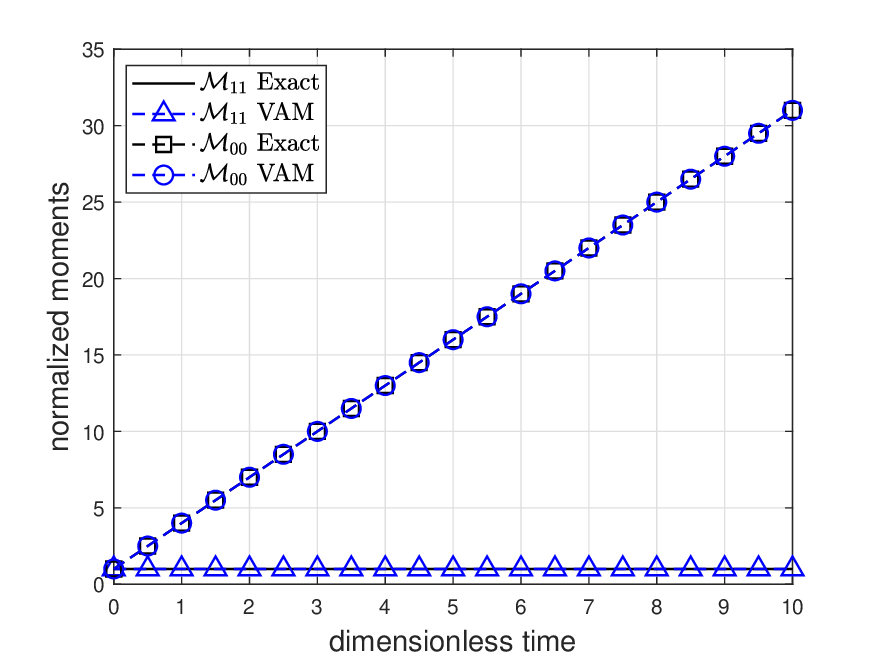}
		\caption{Zeroth and cross moment}
		\label{B_f3_1}
	\end{subfigure}
	\begin{subfigure}{.325\textwidth}
		\centering
		\includegraphics[width=1.00\textwidth]{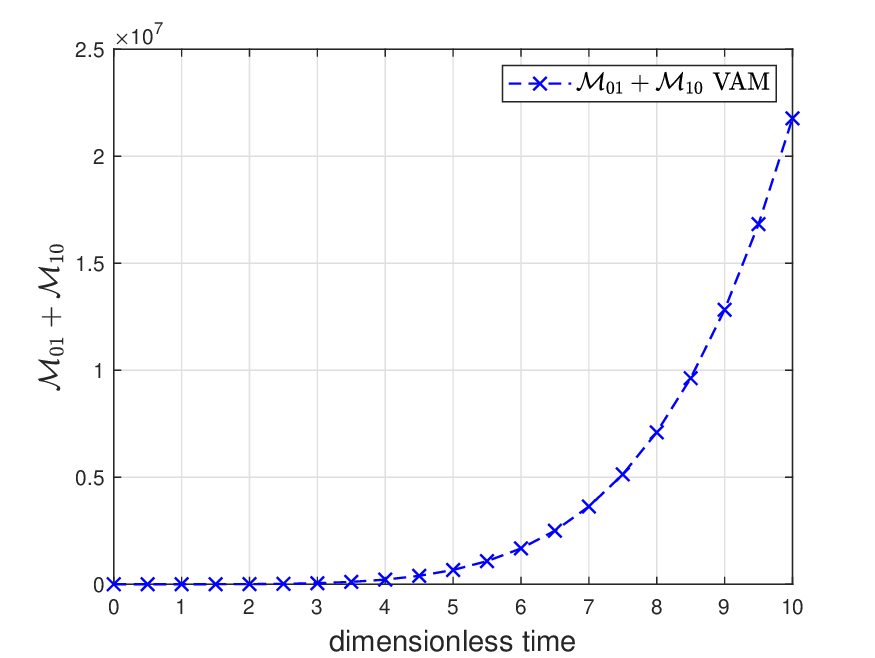}
		\caption{ First Moment}
		\label{B_f3_2}
	\end{subfigure}
	\begin{subfigure}{.325\textwidth}
		\centering
		\includegraphics[width=\textwidth]{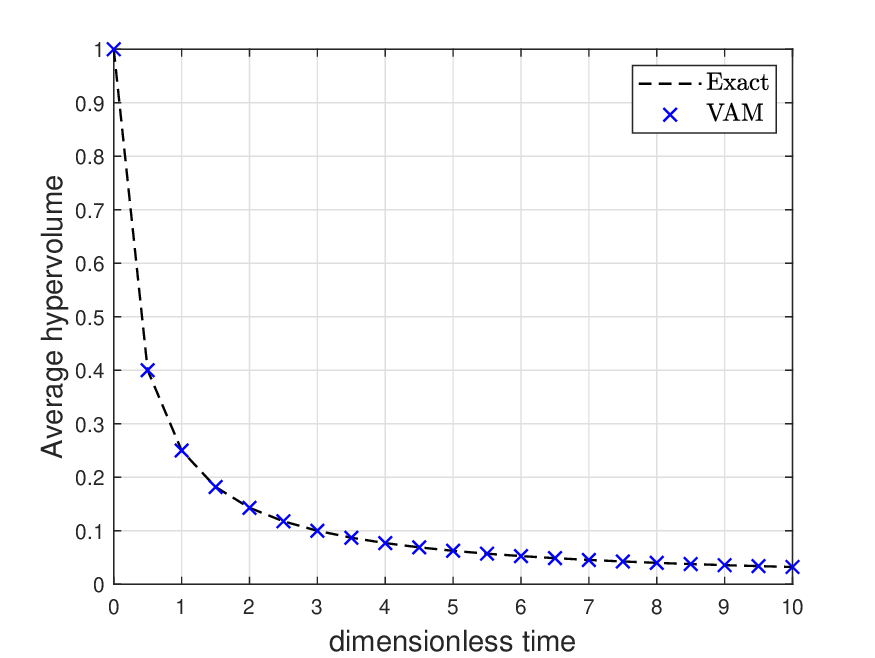}
		\caption{Average hypervolume }
		\label{B_f3_3}
	\end{subfigure}
	\caption{Comparison of different order moments for multi-dimensional model.}
	\label{B_f3}
\end{figure}

\noindent Case $(ii)$: On contrast to previous case, in this case first order mixed moments are conserved and  zeroth and cross moments are calculated exactly. Figure \ref{B_f4_1} presents the zeroth  and first cross moment. The mixed moments are conserved,  when VAM is used as depicted in Figure \ref{B_f4_2}. Finally, average hypervolume is plotted in Figure \ref{B_f4_3}.

\begin{figure}[H]
	\begin{subfigure}{.325\textwidth}
		\centering
		\includegraphics[width=1.00\textwidth]{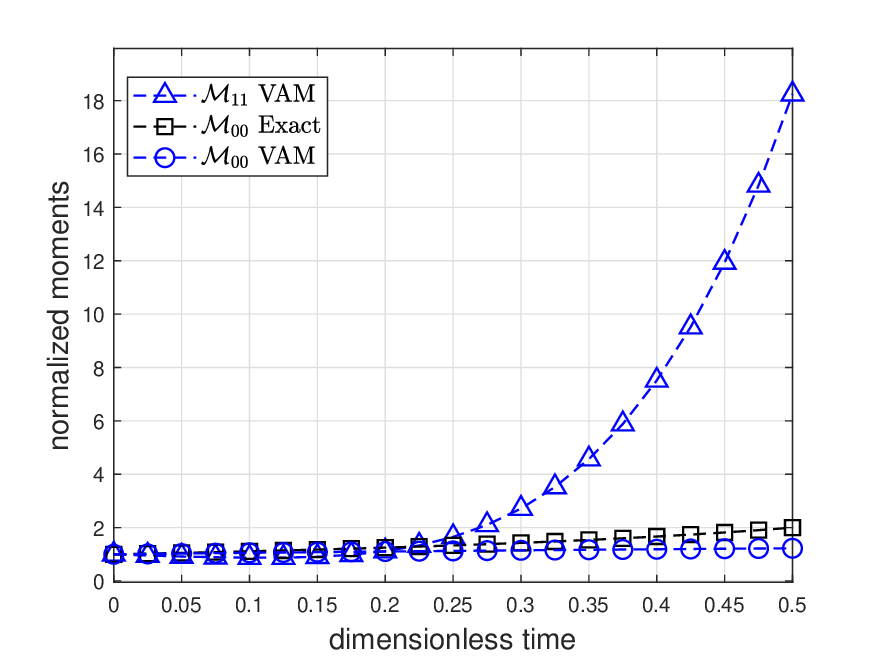}
		\caption{Zeroth and cross moemnt}
		\label{B_f4_1}
	\end{subfigure}
	\begin{subfigure}{.325\textwidth}
		\centering
		\includegraphics[width=1.00\textwidth]{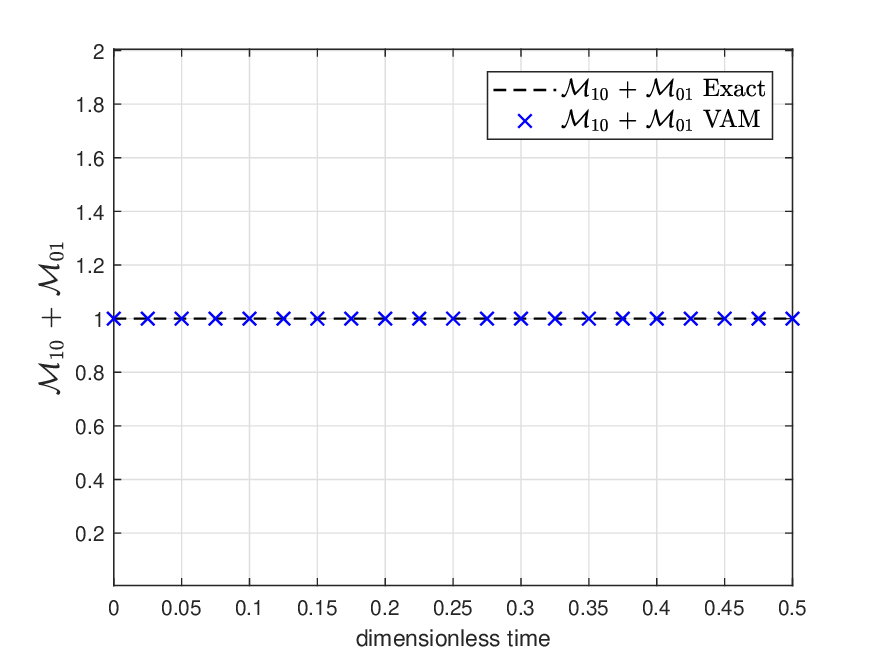}
		\caption{First Moment}
		\label{B_f4_2}
	\end{subfigure}
	\begin{subfigure}{.325\textwidth}
		\centering
		\includegraphics[width=\textwidth]{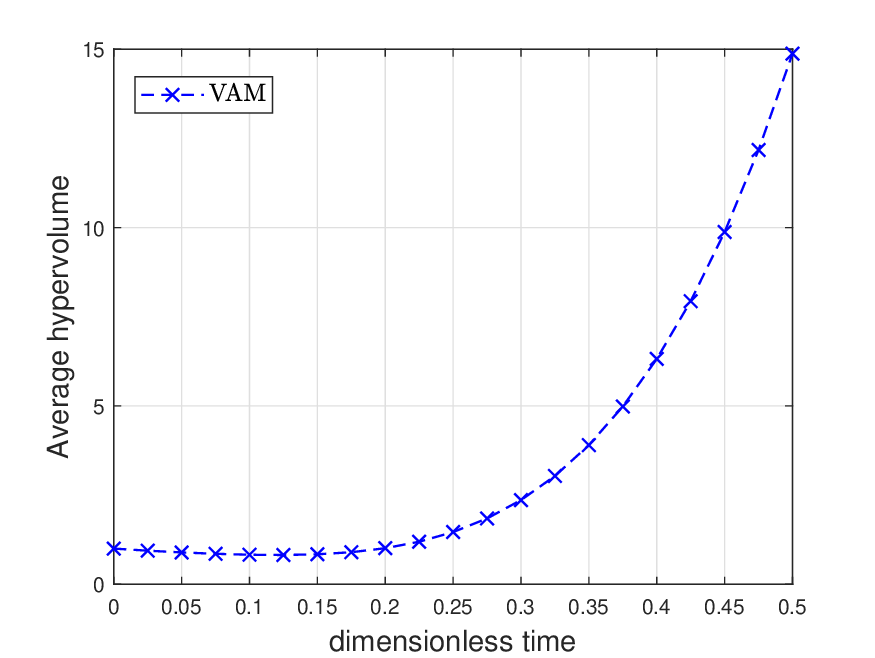}
		\caption{Average hypervolume }
		\label{B_f4_3}
	\end{subfigure}
	\caption{Comparison of different order moments for multi-dimensional model. }
	\label{B_f4}
\end{figure}


\section{Conclusion}\label{sec_6}
This study provided a comprehensive stability and consistency analysis of the  new scheme volume average method for the collisional nonlinear breakage equation. It  documented  the mathematical and numerical results in comparison with the existing scheme fixed pivot techniques. It is noteworthy that the new scheme  achieves first order convergence rate over random grids, whereas the existing scheme fixed pivot techniques  are inconsistent over random grids.  Moreover, it is shown that the new scheme is first order convergence rate over oscillatory grids. Furthermore, our considered equation is extended for two-dimensional case over rectangular grids. 
Most importantly, numerical experiments have been performed successfully for one and two-dimensional  collisional  nonlinear breakage equation with two test examples for each case. Additionally, we explored a detailed error estimation for the new scheme.

\section{Acknowledgement}  PK thanks Ministry of Education (MoE), Govt.\ of India for her funding support during her PhD program. AG thanks University Grants Commission (UGC), Govt. of India for  their funding support during her PhD program. JS thanks to Science and Engineering Research Board (SERB), Govt. of India for their support through Core Research Grant (CRG / 2023 / 001483) during this work.

\section{Conflict of interest statement} All the authors certify that
they do not have any conflict of interest.

\section{Data availability statement } The data that support the findings of this study is available from the corresponding author upon reasonable request.
\bibliographystyle{amsplain} 
\bibliography{bibliography_cat.bib}

\begin{appendices}
	\section{ Derivation of  equation \eqref{2_21} in   Proposition \ref{prop_2_1}: }\label{appen_A}
	We have	to prove the expression 
	\begin{align}
		&\sum_{i=1}^{I}x_i\hat{B}_{i-1}\lambda_i^-(\bar{v}_{i-1})H(\bar{v}_{i-1}-x_{i-1}) +\sum_{i=1}^{I}x_i\hat{B}_{i}
		\lambda_i^-(\bar{v}_{i})H(x_i-\bar{v}_{i})+ \sum_{i=1}^{I}x_i\hat{B}_{i}\lambda_i^+(\bar{v}_{i})H(\bar{v}_{i}-x_{i})
		\notag \\&	+\sum_{i=1}^{I}x_i\hat{B}_{i+1} \lambda_i^+(\bar{v}_{i+1})H(x_{i+1}-\bar{v}_{i+1})=\sum_{i=1}^{I}\bar{v}_i\hat{B}_{i}
	\end{align}
	\begin{proof}
		\begin{align}
			\frac{\dd}{\dd t}\sum_{i=1}^{I}x_i\hat{N}_i(t)=&\sum_{i=1}^{I}x_i\hat{B}_{i-1}\lambda_i^-(\bar{v}_{i-1})H(\bar{v}_{i-1}-x_{i-1}) +\sum_{i=1}^{I}x_i\hat{B}_{i}
			\lambda_i^-(\bar{v}_{i})H(x_i-\bar{v}_{i})\notag \\ 
			&+ \sum_{i=1}^{I}x_i\hat{B}_{i}\lambda_i^+(\bar{v}_{i})H(\bar{v}_{i}-x_{i})+\sum_{i=1}^{I}x_i\hat{B}_{i+1} \lambda_i^+(\bar{v}_{i+1})H(x_{i+1}-\bar{v}_{i+1})\notag\\&-\sum_{i=1}^{I}\sum_{j=1}^{I}x_i\K(x_i,x_j)\hat{ N}_i(t)\hat{ N}_j(t).
		\end{align}
		If $v_i>x_i$, then the above equation is written as
		\begin{align}
			\frac{\dd}{\dd t}\sum_{i=1}^{I}x_i\hat{N}_i(t)=&\sum_{i=1}^{I}x_i\hat{B}_{i-1}\lambda_i^-(\bar{v}_{i-1})
			+ \sum_{i=1}^{I}x_i\hat{B}_{i}\lambda_i^+(\bar{v}_{i})-\sum_{i=1}^{I}\sum_{j=1}^{I}x_i\K(x_i,x_j)\hat{ N}_i(t)\hat{ N}_j(t)\notag\\
			=&\sum_{i=1}^{I}x_{i+1}\hat{B}_{i}\lambda_{i+1}^-(\bar{v}_{i})
			+ \sum_{i=1}^{I}x_i\hat{B}_{i}\lambda_i^+(\bar{v}_{i})-\sum_{i=1}^{I}\sum_{j=1}^{I}x_i\K(x_i,x_j)\hat{ N}_i(t)\hat{ N}_j(t)\notag\\
			=&\sum_{i=1}^{I}\left(x_{i+1}\lambda_i^-(\bar{v}_{i})
			+x_i\lambda_i^+(\bar{v}_{i})\right)\hat{B}_{i}-\sum_{i=1}^{I}\sum_{j=1}^{I}x_i\K(x_i,x_j)\hat{ N}_i(t)\hat{ N}_j(t)\notag\\
			=&\sum_{i=1}^{I}\left(x_{i+1}\frac{\bar{v}_i-x_{i}}{x_{i+1}-x_i}
			+x_i\frac{\bar{v}_i-x_{i+1}}{x_{i}-x_{i+1}}\right)\hat{B}_{i}-\sum_{i=1}^{I}\sum_{j=1}^{I}x_i\K(x_i,x_j)\hat{ N}_i(t)\hat{ N}_j(t)\notag\\
			=&\sum_{i=1}^{I}\bar{v}_i\hat{B}_{i}-\sum_{i=1}^{I}\sum_{j=1}^{I}x_i\K(x_i,x_j)\hat{ N}_i(t)\hat{ N}_j(t).
		\end{align}
		By	similar argument for $v_i<x_i$, we have
		\begin{align}
			\frac{\dd}{\dd t}\sum_{i=1}^{I}x_i\hat{N}_i(t)=\sum_{i=1}^{I}\bar{v}_i\hat{B}_{i}-\sum_{i=1}^{I}\sum_{j=1}^{I}x_i\K(x_i,x_j)\hat{ N}_i(t)\hat{ N}_j(t).
		\end{align}
	\end{proof}
\end{appendices}

\end{document}